\documentclass[12pt, reqno]{amsproc}
\usepackage[utf8]{inputenc}
\usepackage[T2A]{fontenc}
\usepackage[english]{babel}

\usepackage[margin = 2.5cm]{geometry}
\usepackage{enumitem}
\usepackage[svgcolors, dvipsnames]{xcolor}
\usepackage{graphicx}
\usepackage[all]{xy}
\usepackage{amssymb, amsmath, amscd, amsthm}
\usepackage[colorlinks=true]{hyperref}
\hypersetup{urlcolor=blue, citecolor=red}

\title[Diffeomorphism groups of Morse-Bott foliations on lens spaces, 2]{Homotopy types of diffeomorphism groups of polar Morse-Bott foliations on lens spaces, 2}

\author{Sergiy Maksymenko}
\address{Algebra and Topology Department, Institute of Mathematics NAS of Ukraine \\
Teresh\-chen\-kivska str., 3, Kyiv, 01024, Ukraine}
\email{maks@imath.kiev.ua}

\keywords{Foliation, diffeomorphism, homotopy type, lens space, solid torus}
\subjclass[2020]{
    57R30, % Foliations in differential topology; geometric theory
    57T20% Homotopy groups of topological groups and homogeneous spaces
}

\newcommand\mycolor[1]{}

\setlist[enumerate]{itemsep=0.3ex, topsep=0.3ex, label={\rm(\arabic*)}}
\setlist[itemize]{itemsep=0.3ex, topsep=0.3ex, leftmargin=4ex}

\newtheorem{theorem}[subsection]{Theorem}
\newtheorem{lemma}[subsection]{Lemma}

\newtheorem{corollary}[subsection]{Corollary}

\newtheorem{remark}[subsection]{Remark}
\newtheorem{example}[subsection]{Example}
\newtheorem{definition}[subsection]{Definition}

\makeatletter
\@addtoreset{subsection}{section}
\@addtoreset{equation}{section}
\@addtoreset{figure}{section}
\@addtoreset{table}{section}
\makeatother

\makeatletter
\newcommand\testshape{family=\f@family; series=\f@series; shape=\f@shape.}
\def\myemphInternal#1{\if n\f@shape%
\begingroup\itshape #1\endgroup\/%
\else\begingroup\sf\itshape\small #1\endgroup%
\fi}
\def\myemph{\futurelet\testchar\MaybeOptArgmyemph}
\def\MaybeOptArgmyemph{\ifx[\testchar \let\next\OptArgmyemph
                 \else \let\next\NoOptArgmyemph \fi \next}
\def\OptArgmyemph[#1]#2{\index{#1}\myemphInternal{#2}}
\def\NoOptArgmyemph#1{\myemphInternal{#1}}
\makeatother
\newcommand\xmonoArrow[1]{\lhook\joinrel\xrightarrow{~#1~}}

\newcommand\xepiArrow[1]{\xrightarrow{#1}\!\!\!\!\!\to}

\newcommand\amatr[4]{\left(\!\begin{smallmatrix}#1\ &#2\\[0.5mm] #3\ &#4 \end{smallmatrix}\!\right)}
\newcommand\avect[2]{\left(\begin{smallmatrix}#1 \\ #2 \end{smallmatrix}\right)}
\newcommand\ddd[2]{\tfrac{\partial#1}{\partial#2}}
\newcommand\Aman{A}
\newcommand\Bman{B}
\newcommand\Cman{C}

\newcommand\Eman{E}

\newcommand\Lman{L}
\newcommand\Mman{M}
\newcommand\Nman{N}

\newcommand\Pman{P}
\newcommand\Qman{Q}

\newcommand\Uman{U}
\newcommand\Vman{V}
\newcommand\Wman{W}
\newcommand\Xman{X}

\newcommand\bC{\mathbb{C}}
\newcommand\bD{\mathbb{D}}

\newcommand\bR{\mathbb{R}}
\newcommand\bZ{\mathbb{Z}}
\newcommand\id{\mathrm{id}}          % identity map    
\newcommand\eps{\varepsilon}                   % epsilon
\newcommand\nrm[1]{\vert#1\vert}
\newcommand\dnrm[1]{\vert#1\vert}
\newcommand\restr[2]{#1\vert_{#2}}
\newcommand\GL{\mathrm{GL}}

\newcommand\SO{\mathrm{SO}}
\newcommand\Ort{\mathrm{O}}

\newcommand\Aut{\mathrm{Aut}}       % automorphisms
\newcommand\Diff{\mathcal{D}}       % diffeomorphisms
\newcommand\DiffPl{\Diff^{+}}       % diffeomorphisms preserving orientation
\newcommand\Isom{\mathrm{Isom}}     % group of isometries
\newcommand\Stab{\mathcal{S}}       % stabilizer
\newcommand\RP[1]{\mathbb{RP}^{#1}}

\newcommand\Dih{\mathrm{Dih}}       % Dihedral extension 
\newcommand\Cr[1]{\mathcal{C}^{#1}}
\newcommand\Cinfty{\mathcal{C}^{\infty}}

\newcommand\Ci[2]{\mathcal{C}^{\infty}(#1,#2)}               % space of C^\infty maps
\newcommand\Stabilizer[1]{\Stab(#1)}

\newcommand\fixsymbol{\mathrm{fix}}%{\!\times}
\newcommand\invsymbol{}
\newcommand\nbsymbol{\mathrm{nb}}
\newcommand\folsymbol{*}%{\approx}

\newcommand\DiffInv[3][\empty]{\Diff_{\invsymbol}(#2,#3\ifx\empty #1\relax\else,#1\fi)}

\newcommand\DiffFix[3][\empty]{\Diff_{\fixsymbol}(#2,#3\ifx\empty #1\relax\else,#1\fi)}

\newcommand\DiffNb[3][\empty]{\Diff_{\nbsymbol}(#2,#3\ifx\empty #1\relax\else,#1\fi)}

\newcommand\DiffHFix[3][\empty]{\Diff^{0}_{\fixsymbol}(#2,#3\ifx\empty #1\relax\else,#1\fi)}

\newcommand\DiffHNb[3][\empty]{\Diff^{0}_{\nbsymbol}(#2,#3\ifx\empty #1\relax\else,#1\fi)}

\newcommand\DiffPlusFix[3][\empty]{\Diff^{+}_{\fixsymbol}(#2,#3\ifx\empty #1\relax\else,#1\fi)}

\newcommand\FDiff[2][\empty]{\Diff^{\folsymbol}(#2\ifx\empty #1\relax\else,#1\fi)}
\newcommand\FDiffFix[2][\empty]{\Diff^{\folsymbol}_{\fixsymbol}(#2\ifx\empty #1\relax\else,#1\fi)}
\newcommand\FDiffA[2][\empty]{\Diff^{=}(#2\ifx\empty #1\relax\else,#1\fi)}

\newcommand\VBAut[2][\empty]{\GL(#2\ifx\empty #1\relax\else,#1\fi)}

\newcommand\DiffLP{\Diff}  % leaf preserving diffeomorphisms
\newcommand\DiffLPInv[3][\empty]{\DiffLP_{inv}(#2,#3\ifx\empty#1\relax\else,#1\fi)}

\newcommand\DiffLPFix[3][\empty]{\DiffLP_{fix}(#2,#3\ifx\empty#1\relax\else,#1\fi)}

\newcommand\DiffLPNb[3][\empty]{\DiffLP_{nb}(#2,#3\ifx\empty#1\relax\else,#1\fi)}

\newcommand\func{f}
\newcommand\gfunc{g}
\newcommand\dif{h}
\newcommand\gdif{g}

\newcommand\rphi{{\mycolor{red}\phi}}

\newcommand\px{x}
\newcommand\py{y}
\newcommand\pz{z}
\newcommand\pu{u}
\newcommand\pv{v}
\newcommand\pw{w}
\newcommand\pui[1]{\pu_{#1}}
\newcommand\pvi[1]{\pv_{#1}}
\newcommand\pxi[1]{\px_{#1}}

\newcommand\pyi[1]{\py_{#1}}

\newcommand\Circle{S^1}
\newcommand\UInt{[0;1]}      % unit interval [0;1]

\newcommand\FolDiffxi[1]{\FolDiff(\AFoliation_{#1})}
\newcommand\FolDiffPlxi[1]{\FolDiffPl(\AFoliation_{#1})}
\newcommand\FolLpDiffxi[1]{\FolLpDiff(\AFoliation_{#1})}

\newcommand\term[2][\empty]{\myemph[#1]{#2}}

\newcommand\DiffM{\Diff(\Mman)}
\newcommand\DiffMX{\Diff(\Mman,\Xman)}

\newcommand\DiffR{\Diff(\bR)}

\newcommand\DiffA{\Diff(\Aman)}
\newcommand\DiffPlA{\DiffPl(\Aman)}

\newcommand\DiffI{\Diff(\UInt)}
\newcommand\DiffPlI{\Diff^{+}(\UInt)}

\newcommand\DiffAM{\DiffA\times\DiffM}

\newcommand\DiffRM{\DiffR\times\DiffM}

\newcommand\CiMR{\Ci{\Mman}{\bR}}

\newcommand\StabLR[1]{\Stab_{\text{\sf LR}}(#1)}   % left-right stabilizer
\newcommand\StabR[1]{\Stab_{\text{\sf R}}(#1)}     % right stabilizer

\newcommand\StabLRImg[1]{\Stab_{\text{\sf LR}}^{img}(#1)}
\newcommand\StabLRImgPl[1]{\Stab_{\text{\sf LR}}^{img+}(#1)}

\newcommand\SMRf{\StabLR{\func}}
\newcommand\SMAf{\StabLRImg{\func}}
\newcommand\SPlMAf{\StabLRImgPl{\func}}

\newcommand\SMf{\StabR{\func}}

\newcommand\SMAfd{\StabLRImg{\func,\partial\ATor}}
\newcommand\SMfd{\StabR{\func,\partial\ATor}}

\newcommand\CinftyEeven{\Cinfty_{\mathrm{ev}}}  % space of even functions

\newcommand\ATor{\mathbf{T}}  % solid torus
\newcommand\AOTor{\mathbf{U}} % interior of the solid torus

\newcommand\gxi{{\mycolor{blue}\xi}} % gluing map

\newcommand\Lpq[2]{L_{#1,#2}} % Lens space Lpq

\newcommand\FolDiff{\Diff^{fol}}    % foliated diffeomorphisms
\newcommand\FolDiffPl{\FolDiff_{+}} % foliated diffeomorphisms preserving some "direction"
\newcommand\FolLpDiff{\Diff^{lp}}   % leaf-preserving diffeomorphisms

\newcommand\Foliation{\mathcal{F}}
\newcommand\GFoliation{\mathcal{G}}
\newcommand\AFoliation{\Foliation}
\newcommand\BFoliation{\GFoliation}
\newcommand\leaf[1]{\Lman_{#1}}

\newcommand\vbp{{\mycolor{blue}p}}  % projection
\newcommand\al{{\mycolor{red}w}}
\newcommand\az{{\mycolor{red}\pz}}

\newcommand\difmatr[1]{\mathsf{#1}}
\newcommand\diftor[1]{\gdif_{#1}}
\newcommand\diflpq[1]{\widehat{#1}}

\newcommand\mXimatr{\difmatr{X}}
\newcommand\mAmatr{\difmatr{A}}

\newcommand\mDtwist{\difmatr{D}}
\newcommand\mLambda{\difmatr{\Lambda}}
\newcommand\mMu{\difmatr{M}}
\newcommand\mTau{\difmatr{T}}

\newcommand\hDtwist{d}
\newcommand\hLambda{\lambda}
\newcommand\hMu{\mu}
\newcommand\hTau{\tau}
\newcommand\hRho[1]{\rho_{#1}}  % rotation

\newcommand\lDtwist{\diflpq{\hDtwist}}
\newcommand\lLambda{\diflpq{\hLambda}}
\newcommand\lMu{\diflpq{\hMu}}
\newcommand\lTau{\diflpq{\hTau}}
\newcommand\lRot[1]{\diflpq{\hRho{#1}}}
\newcommand\lTheta{\diflpq{\theta}}
\newcommand\lSigmaPl{\sigma_{+}}
\newcommand\lSigmaMin{\sigma_{-}}

\newcommand\DiffATor{\Diff(\ATor)}
\newcommand\DiffdATor{\Diff(\ATor, \partial\ATor)}

\newcommand\FolDiffdATor{\FolDiff(\AFoliation, \partial\ATor)}     % foliated diffeomorphisms
\newcommand\FolLpDiffdATor{\FolLpDiff(\AFoliation, \partial\ATor)} % leaf preserving diffeomorphisms

\newcommand\DiffAFol{\FolDiff(\AFoliation)}
\newcommand\DiffPlAFol{\FolDiffPl(\AFoliation)}
\newcommand\DiffAFolLp{\FolLpDiff(\AFoliation)}

\newcommand\MPGSolidTorus{\mathcal{G}} % \pi_0 Diff(S^1 x D^2)

\newcommand\RotSub{\mathcal{R}}        % rotation subgroup
\newcommand\AffDiffATor{\mathcal{T}}   % R x G - semidirect product in GL(2,Z) -  "affine" subgroup of Diff(S^1 x D^2)

\newcommand\ADiffLpq[1]{\mathcal{L}_{#1}}
\newcommand\AFolDiffLpq[1]{\widehat{\mathcal{L}}_{#1}}

\newcommand\FolDiffLpq[2]{\FolDiff(\AFoliation_{#1,#2})}
\newcommand\FolDiffPlLpq[2]{\FolDiffPl(\AFoliation_{#1,#2})}
\newcommand\FolLpDiffLpq[2]{\FolLpDiff(\AFoliation_{#1,#2})}

\newcommand\PAFoliation{\AFoliation}

\newcommand\BLman{\mathbf{\Lman}}
\newcommand\PBLman{\BLman_{\gxi}}

\newcommand\xtor[2]{{}^{#1}_{#2}}

\newcommand\xar[2]{ \ar@{=>}[r]^-{#1}_-{#2}}

\newcommand\diflpdir[4][1.4em]{\xymatrix@C=#1{#2: \xtor{0}{1} \xar{#3}{#4} & \xtor{0}{1}}}
\newcommand\diflprew[4][1.4em]{\xymatrix@C=#1{#2: \xtor{0}{1} \xar{#3}{#4} & \xtor{1}{0}}}

\newcommand\prM{p_2}
\newcommand\prA{p_1}

\newcommand\JProp{{\rm(J)}} % property (J)
\newcommand\Calg[2]{\mathcal{C}_{#1}^{\infty}(#2)}  % algebra of germs
\newcommand\Jideal[2]{J_{#1}(#2)} % Jacobi ideal

\newcommand\CiTwoFol{\Cinfty(\AFoliation_0,\AFoliation_1)}
\newcommand\DiffTwoFol{\Diff(\AFoliation_0,\AFoliation_1)}

\newcommand\jjj[1]{#1_{1}}
\newcommand\iii[1]{#1_{0}}

\newcommand\gA{A}
\newcommand\gB{B}
\newcommand\gC{C}

\newcommand\rt[2]{\!\sqrt[#2]{#1}}
\newcommand\ophi{{\mycolor{green}q}}
\newcommand\prC{p}
\newcommand\sincl{j}

\newcommand\smprM{{\mycolor{red}{\sigma}}}
\newcommand\smprA{{\mycolor{red}{\theta}}}

\begin{document}
\begin{abstract}
Let $\mathcal{F}$ be a Morse-Bott foliation on the solid torus $T=S^1\times D^2$ into $2$-tori parallel to the boundary and one singular central circle.
Gluing two copies of $T$ by some diffeomorphism between their boundaries, one gets a lens space $L_{p,q}$ with a Morse-Bott foliation $\mathcal{F}_{p,q}$ obtained from $\mathcal{F}$ on each copy of $T$ and thus consisting of two singluar circles and parallel $2$-tori.
In the previous paper [O.~Khokliuk, S.~Maksymenko, Journ. Homot. Rel. Struct., 2024, 18, 313-356] there were computed weak homotopy types of the groups $\mathcal{D}^{lp}(\mathcal{F}_{p,q})$ of leaf preserving (i.e.\ leaving invariant each leaf) diffeomorphisms of such foliations.
In the present paper it is shown that the inclusion of these groups into the corresponding group $\mathcal{D}_{+}^{fol}(\mathcal{F}_{p,q})$ of foliated (i.e.\ sending leaves to leaves) diffeomorphisms which do not interchange singular circles are homotopy equivalences.
\end{abstract}

\maketitle

\section{Introduction}
The paper is devoted to computations of homotopy types of diffeomorphism groups of a solid torus and lens spaces preserving foliations by level sets of ``simplest'' Morse-Bott functions whose sets of critical points consist of extreme circles only.

Note that the homotopy types of diffeomorphism groups of compact manifolds in dimensions $1,2$ are described completely, \cite{Smale:ProcAMS:1959, EarleEells:JGD:1969, EarleSchatz:DG:1970, Gramain:ASENS:1973}; in dimension $3$ there is a lot of information, e.g.~\cite{Hatcher:AnnM:1983, Gabai:JDG:2001, HongKalliongisMcCulloughRubinstein:LMN:2012}; while in dimensions $\geq4$ only very specific cases are computed, e.g.~\cite{Novikov:IzvAN:1965, Schultz:Top:1971, Hajduk:BP:1978, DwyerSzczarba:IJM:1983, Kupers:GT:2019, BerglundMadsen:AM:2020}.
Computations for the groups of leaf preserving diffeomorphisms of foliations are usually related with perfectness properties of such groups proved independently by T.~Rybicki~\cite{Rybicki:MonM:1995} and T.~Tsuboi~\cite{Tsuboi:Fol:2005}, which extended the results by M.-R.~Herman~\cite{Herman:CR:1971}, W.~Thurston~\cite{Thurston:BAMS:1974}, J.~Mather~\cite{Mather:Top:1971, Mather:BAMS:1974}, D.~B.~A.~Epstein~\cite{Epstein:CompMath:1970} on simplicity of groups of compactly supported diffeomorphisms isotopic to the identity.

For foliations with singularities there is much less information, e.g.~\cite{Fukui:JMKU:1980, Rybicki:DM:1998, Maksymenko:OsakaJM:2011, LechMichalik:PMD:2013}.
On the other hand, the most relevant references to the present paper are results by K.~Fukui, S.~Ushiki~\cite{FukuiUshiki:JMKU:1975} and K.~Fukui~\cite{Fukui:JJM:1976} on the homotopy type of regular foliations on $3$-manifolds with finitely many Reeb components.

Just to illustrate what is a difficulty with singular foliations consider a simple example.
Let $\func\colon\bR^{3}\to\bR$ be a function, $\mathcal{F}$ a partition into the level sets of $\func$, and $\dif=(a,b,c)\colon\bR^{3}\to\bR^{3}$ a diffeomorphism leaving invariant each leaf of $\mathcal{F}$.

First assume $\func(x,y,z)=z$.
Then $\mathcal{F}$ is a usual ``regular'' foliation consisting of planes parallel to the $xy$-plane, and $\dif$ should be given by a formula $\dif(x,y,z)=\bigl(a(x,y,z),b(x,y,z),z\bigr)$, where for each fixed $z\in\bR$ the map $(x,y)\mapsto\bigl(a(x,y,z),b(x,y,z)\bigr)$ is a diffeomorphism of $\bR^{2}$.
In other words, $\dif$ can be regarded as a one-parametric family of diffeomorphisms of $\bR^{2}$, or an isotopy of $\bR^{2}$, and even as a map of $\bR$ into the group of diffeomorphisms of $\bR^{2}$.
In particular, the information of the homotopy type of $\Diff(\bR^{2})$ might be useful in this case.

On the other hand, suppose that $\func(x,y,z)=x^2+y^2$.
In this case $\mathcal{F}$ consists of vertical cylinders (being regular leaves) and the $z$-axis (being a singular leaf).
Then the assumption that $\dif$ preserves level sets of $\func$ means that the coordinate functions of $\dif$ satisfy the following identity: $a(x,y,z)^2+b(x,y,z)^2=x^2+y^2$.
The problem of describing deformation properties of such diffeomorphisms seems to be much more complicated.

Given smooth manifolds $\Mman,\Pman$ there is a natural action $\Ci{\Mman}{\Pman} \times \Diff(\Mman) \to \Ci{\Mman}{\Pman}$, $(\func,\dif) \mapsto \func\circ\dif$, of the group $\Diff(\Mman)$ of diffeomorphisms of $\Mman$ on the space of smooth maps $\Ci{\Mman}{\Pman}$, e.g.\ \cite{Mather_1:AnnMath:1968, Sergeraert:ASENS:1972, MondNunoBallesteros:Sing:2020}, see also Section~\ref{sect:left-right-actions}.
Then for each $\func\in\Ci{\Mman}{\Pman}$ its stabilizer
\[
    \Stabilizer{\func} = \{\dif\in\Diff(\Mman)\mid \func\circ\dif=\func\}
\]
consists of diffeomorphisms leaving invariant each level set of $\func$, i.e.\ preserving a singular foliation by level sets of $\func$.

In a series of papers by the author, for the case when $\Mman$ is a compact surface and $\Pman$ is a one-dimensional manifold, i.e.\ $\Pman=\bR$ or $\Circle$, there were studied a ``large'' subset $\mathcal{F}(\Mman,\Pman) \subset \Ci{\Mman}{\Pman}$ of maps $\func\colon\Mman\to\Pman$ with isolated critical points (which includes all Morse maps).
It was shown that for those $\func\in\mathcal{F}(\Mman,\Pman)$ the path components of the stabilizer $\Stabilizer{\func}$ are either contractible or homotopy equivalent to the circle (the latter happens only for few types of maps), see~\cite{Maksymenko:TA:2003,Maksymenko:AGAG:2006,Maksymenko:UMZ:ENG:2012,Maksymenko:TA:2020} for details, references, and applications to the homotopy types of the corresponding orbits of such functions, and~\cite{LeygonieBeers:JACT:2023} for applications to persistent homologies.
However, the technique of those papers is essentially based on the fact that the leaves of those singular foliations are one-dimensional: this allowed to regard them as orbits of some vector fields (at least for oriented $\Mman$).

In a recent series of joint papers by the author with O.~Khokhliuk~\cite{KhokhliukMaksymenko:IndM:2020,KhokhliukMaksymenko:PIGC:2020, KhokhliukMaksymenko:2022, KhokhliukMaksymenko:JHRS:2023} there were developed several techniques for computations of homotopy types of diffeomorphism groups of Morse-Bott and more general classes of ``singular'' foliations in higher dimensions.
In particular, the previous paper \cite{KhokhliukMaksymenko:JHRS:2023} computes the homotopy types of diffeomorphism groups of leaf preserving diffeomorphisms for mentioned above foliations of the solid torus and lens spaces.
In the present paper we find the relationship between those groups and the groups of diffeomorphisms sending leaves to leaves, see Theorems~\ref{th:Dlp_Dfol_he}, \ref{th:DiffLpq_fol_homtype_a}, \ref{th:hom_type_solid_torus_groups}, \ref{th:DiffLpq_fol_homtype} below.

\subsection*{Some definitions and notations}
Let $\AFoliation$ be a partition of a set $\Mman$.
Then a map $\dif\colon\Mman\to\Mman$ is \term{$\AFoliation$-foliated} if for each leaf $\omega\in\AFoliation$ its image $\dif(\omega)$ is a (possibly distinct from $\omega$) leaf of $\AFoliation$ as well.
Also, $\dif$ is \term{$\AFoliation$-leaf preserving}, if $\dif(\omega)=\omega$ for each leaf $\omega\in\AFoliation$.
More generally, if $\BFoliation$ is a partition of a set $\Nman$, then a map $\dif\colon\Mman\to\Nman$ is \term{$(\AFoliation,\BFoliation)$-foliated}, if for each leaf $\omega\in\AFoliation$ its image $\dif(\omega)$ is a leaf of $\BFoliation$.

All manifolds and their diffeomorphisms are assumed to be of class $\Cinfty$.
If $\Mman$ is a manifold, and $\Xman\subset\Mman$ is a subset, then we will denote by $\Diff^{fol}(\AFoliation,\Xman)$, (resp.\ $\Diff^{lp}(\AFoliation,\Xman)$), the groups of $\AFoliation$-foliated, (resp., $\AFoliation$-leaf preserving), $\Cinfty$ diffeomorphisms of $\Mman$ endowed with the corresponding strong $\Cinfty$ Whitney topologies.
If $\Xman=\varnothing$, then we will omit it from notation and denote the above groups by $\Diff^{fol}(\AFoliation)$ and $\Diff^{lp}(\AFoliation)$ respectively.

Throughout the paper $D^2 = \{ \nrm{\az}\leq 1\}$ is the unit disk in the complex plane and $\Circle = \partial D^2$ the unit circle.
For an abelian group $A$ we will denote by $\Dih(A)$ the \term{dihedral extension of $A$}, i.e.\ the semidirect product $A\rtimes\bZ_2$ corresponding to the natural action of $\bZ_2$ on $A$ generated by the isomorphism $o:A\to A$, $o(a)=-a$.
For example, $\Dih(\bZ_n)$ is the dihedral group $\mathbb{D}_n$, and $\Dih(\SO(2))=\Ort(2)$.

\section{Main results}\label{sect:main_result:all}
In this section we will formulate two principal results (Theorems~\ref{th:Dlp_Dfol_he} and~\ref{th:DiffLpq_fol_homtype_a}) concerning ``simplest'' Morse-Bott foliations on a solid torus and lens spaces.
More detailed statements will be presented throughout the paper.

\subsection*{Solid torus}
Let $\ATor = \Circle\times D^2$ be the solid torus.
Consider the following Morse-Bott function $\func\colon\ATor\to\bR$, $\func(\al,\az) = \nrm{\az}^2$, and let $\AFoliation = \{\func^{-1}(t) \mid t\in\UInt \}$ be the partition of $\ATor$ into the level sets of $\func$.
Then $\func^{-1}(0) = \Circle\times 0$ is the ``central circle'' of $\ATor$ consisting of all critical points of $\func$, and thus being a non-degenerate critical submanifold.
For all other values $t\in(0;1]$, $\func^{-1}(t) = \Circle\times\{\nrm{\az}^2=t\}$ is a $2$-torus (product of two circles) ``parallel'' to $\partial\ATor$.
In particular, $\func^{-1}(1) = \partial\ATor$.

N.~Ivanov~\cite{Ivanov:ST:1982} proved that the group $\DiffdATor$ of all diffeomorphism of $\ATor$ fixed on the boundary is contractible.
Also in a recent joint paper~\cite{KhokhliukMaksymenko:JHRS:2023} of the author with O.~Khokhliuk it is computed the weak homotopy type of $\DiffAFolLp$ and shown that all homotopy groups of $\FolLpDiffdATor$ vanish.
One of the main results of this paper accomplishes the results of~\cite{KhokhliukMaksymenko:JHRS:2023} with the following theorem:
\begin{theorem}\label{th:Dlp_Dfol_he}
The pair of groups $\bigl(\DiffAFolLp, \ \FolLpDiffdATor\bigr)$ is a strong deformation retract of the pair $\bigl(\DiffAFol, \ \FolDiffdATor\bigr)$.
\end{theorem}
The proof identifies the groups $\DiffAFol$ and $\DiffAFolLp$ with the stabilizers of $\func$ with respect to the natural ``left-right'' and ``right'' actions of diffeomorphism groups of $\ATor$ and $\UInt$ on $\Ci{\ATor}{\UInt}$, see Section~\ref{sect:left-right-actions}, and then uses the author's result from~\cite{Maksymenko:BSM:2006} claiming that those stabilizers are homotopy equivalent.
The above identification exploits the well-known Whithey theorem on even smooth functions, see Lemma~\ref{lm:Whitney_even_func}.
In fact, it will be established a more general statement concerning smooth fiber-wise definite homogeneous functions of degree $2$ on total spaces of vector bundles, see Theorem~\ref{th:def_2_hom_func}.

Notice that the mapping class group $\pi_0\Diff(\ATor)$ can be identified with some subgroup $\mathcal{T}$ of $\GL(2,\bZ)$, and the natural projection $\Diff(\ATor) \to \pi_0\Diff(\ATor) \cong \mathcal{T}$ admits a section homomorphism $s\colon\mathcal{T} \to \Diff(\ATor)$.
As a consequence of results from~\cite{Ivanov:ST:1982, KhokhliukMaksymenko:JHRS:2023}, we obtain, see Theorem~\ref{th:hom_type_solid_torus_groups}, that the following inclusions are weak homotopy equivalences:
\begin{gather*}
    \mathcal{T} \ \xmonoArrow{s} \ \DiffAFolLp     \ \subset \ \DiffAFol     \ \subset \ \Diff(\ATor), \\
    \{\id\}      \ \subset \ \FolLpDiffdATor \ \subset \ \FolDiffdATor \ \subset \ \Diff(\ATor,\partial\ATor).
\end{gather*}
In particular, the latter groups are weakly contractible.
As an application, we see that every principal $\FolLpDiffdATor$- or $\FolDiffdATor$-fibration over any CW-complex has a section, and therefore it is trivial.
Moreover, the structure group of every principal $\DiffAFolLp$- or $\DiffAFol$-fibration reduces to the discrete group $\mathcal{T}$, so the corresponding $\mathcal{T}$-bundle is a covering map.

\subsection*{Lens spaces}
Recall that any $3$-manifold $\Lman_{\gxi}$ obtained by gluing two copies of the solid torus $\ATor$ by some diffeomorphism of their boundaries $\gxi\colon\partial\ATor\to\partial\ATor$ is called a \term{lens space}, see e.g.~\cite{Reidemeister:AMSUH:1935, Brody:AM:1960, Bonahon:Top:1983, Gadgil:GT:2001, HongKalliongisMcCulloughRubinstein:LMN:2012} for basic properties of lens spaces.

More precisely, a lens space can be defined as follows.
Let $\Lman:=\ATor\times\{0,1\}$ be two copies of $\ATor$.
Identify each point $(\px,0)$ with $(\gxi(\px),1)$ and denote the obtained quotient space by $\Lman_{\gxi}$.
Let also $\vbp\colon\Lman \to \Lman_{\gxi}$ be the corresponding quotient map, $\ATor_i := \vbp(\ATor\times\{i\})$, $i=0,1$, and $\Cman_i := \Circle\times 0 \times i$ be the ``central circle'' of the torus $\ATor_i$.

As just noted above, $\ATor$ has a Morse-Bott foliation $\AFoliation$ into the central circle and $2$-tori ``parallel'' to the boundary.
In particular, $\partial\ATor$ is a leaf of $\AFoliation$.
Since $\gxi$ identifies $\partial\ATor\times\{0\}$ with $\partial\ATor\times\{1\}$, it induces the foliation $\AFoliation_{\gxi}$ on $\Lman_{\gxi}$ whose leaves are the images of the corresponding leaves of foliations on those tori.
In particular, $\AFoliation_{\gxi}$ has two singular leaves (the central circles) $\Cman_0$ and $\Cman_1$ and all other leaves are $2$-tori parallel each other.
We will call this foliation \term{polar}.
It appears very often in studying lens spaces, due to the effect observed by F.~Bonahon~\cite{Bonahon:Top:1983} that each of those tori is preserved up to isotopy by diffeomorphisms of $\Lman_{\gxi}$.

One can also define a Morse-Bott function $\gfunc\colon \Lman_{\gxi}\to\bR$ such that $\AFoliation_{\gxi}$ coincides with the partition of $\Lman_{\gxi}$ into the level sets of $\gfunc$.
For example, define the function $\widehat{\gfunc}:\Lman\to[0;2]$ by
\[
    \widehat{\gfunc}(\al,\az) =
    \begin{cases}
        \nrm{\az}^2,   & (\al,\az)\in\ATor\times\{0\},\\
        2-\nrm{\az}^2, & (\al,\az)\in\ATor\times\{1\}.
    \end{cases}
\]
Then $\widehat{\gfunc}$ induces a Morse-Bott $\Cinfty$ function $\gfunc\colon\Lman_{\gxi}\to[0;2]$ such that $\widehat{\gfunc} = \gfunc\circ\vbp$ and the set of critical points of $\gfunc$ is a union of two circles $\gfunc^{-1}(0)=\Cman_0$ and $\gfunc^{-1}(1)=\Cman_1$.
Evidently, $\ATor_0=\gfunc^{-1}\bigl([0;1]\bigr)$, $\ATor_1=\gfunc^{-1}\bigl([1;2]\bigr)$, so $\partial\ATor_0=\partial\ATor_1=\gfunc^{-1}(1)$.

Let $\FolDiffxi{\gxi}$ be the group of $\AFoliation_{\gxi}$-foliated diffeomorphisms of $\Lman_{\gxi}$, $\FolLpDiffxi{\gxi}$ be its (normal) subgroup consisting of $\AFoliation_{\gxi}$-leaf preserving ones, and
\[
    \FolDiffPlxi{\gxi} = \{  \dif\in\FolLpDiffxi{\gxi} \mid \dif(\Cman_i)=\Cman_i, i=0,1 \}
\]
be the another normal subgroup of $\FolDiffxi{\gxi}$ leaving invariant each central circle $\Cman_i$.
Then we have the following inclusions:
\begin{equation}\label{equ:diag:Lpq_all_groups:line}
    \FolLpDiffxi{\gxi}   \ \subset \
    \FolDiffPlxi{\gxi}   \ \subset \
    \FolDiffxi{\gxi}.
\end{equation}
Evidently, if $\FolDiffPlxi{\gxi}$ does not coincide with all the group $\FolDiffxi{\gxi}$, then $\FolDiffPlxi{\gxi}$ has index $2$ in $\FolDiffxi{\gxi}$.
Our second result is the following:
\begin{theorem}\label{th:DiffLpq_fol_homtype_a}
The group $\FolLpDiffxi{\gxi}$ is a strong deformation retract of $\FolDiffPlxi{\gxi}$.
\end{theorem}
In fact, we will prove it for a more general class of foliations by level sets of Morse-Bott functions on a closed manifold having only extreme critical submanifolds, see Theorem~\ref{th:DFLlp_DPlFLfol}.
Moreover, using previous results we will also deduce homotopy types of all the groups in~\eqref{equ:diag:Lpq_all_groups:line}, see Theorem~\ref{th:DiffLpq_fol_homtype}.

\subsection*{Structure of the paper}
Section~\ref{sect:prelim} contains two technical statements about sections of epimorphisms of topological groups and famous Whitney theorem on even functions, see Lemma~\ref{lm:Whitney_even_func}.
Section~\ref{sect:left-right-actions} collects preliminary results about left-right actions of groups of diffeomorphisms.
Section~\ref{sect:fib_hom_func_vb} is devoted to properties of fiberwise homogeneous functions on the total spaces of vector bundles.
In particular, in Section~\ref{sect:fib_hom_func_vb} we prove Theorem~\ref{th:def_2_hom_func} including Theorem~\ref{th:Dlp_Dfol_he} as a particular case.
The proof uses a mentioned above Whitney theorem.
Further in Section~\ref{sect:proof:th:DiffLpq_fol_homtype} we consider high-dimensional analogues of lens spaces obtained by gluing unit disk bundles over some manifolds by some diffeomorphism of their boundaries (assuming that such a diffeomorphism exists).
We prove there Theorem~\ref{th:DFLlp_DPlFLfol} including Theorem~\ref{th:DiffLpq_fol_homtype_a} as a particular case.

In Section~\ref{sect:main_result:solid_torus} we describe the homotopy types of the groups from Theorem~\ref{th:Dlp_Dfol_he}, see Theorem~\ref{th:hom_type_solid_torus_groups}, and in Section~\ref{sect:lens_spaces} describe the homotopy types of the groups in~\eqref{equ:diag:Lpq_all_groups:line} for all lens spaces, see Theorem~\ref{th:DiffLpq_fol_homtype}.

\section{Preliminaries}\label{sect:prelim}
\subsection*{Morphisms of topological groups and monoids}
Let $\sincl\colon\Aman\to\Xman$ be a topological inclusion, i.e.\ a homeomorphism of $\Aman$ onto some subset of $\Xman$.
It will be convenient to say that $\Aman$ is a \term{(strong) deformation retract of $\Xman$ with respect to $\sincl$}, if so is the image $\sincl(\Aman)$.
We will need the following simple lemmas.
\begin{lemma}\label{lm:sect_of_homo}
Let $1 \to \gA \xmonoArrow{\alpha} \gB \xepiArrow{~\prC~} \gC \to 1$ be a short exact sequence of continuous homomorphisms of topological groups in which $\alpha$ is a topological embedding.
Suppose there exists a continuous homomorphism $\smprA\colon\gC\to\gB$ being a section of $\prC$, i.e.\ $\prC\circ\smprA = \id_{\gC}$.
\begin{enumerate}[wide]
\item
Then the map $\zeta\colon\gB \to \gA\times\gC$, $\zeta(b)=\bigl(\smprA(\prC(b^{-1}))\cdot b, \ \prC(b) \bigr)$, is a \term{homeomorphism}.

\item
Suppose that there exists a strong deformation retraction of $\gC$ into the unit $e_{\gC}$ of $\gC$.
Then $\gA$ is a strong deformation retract of $\gB$ with respect to the inclusion $\alpha$.
Suppose, in addition, that $\smprA(\gC)$ is contained in some subgroup $\gB'\subset \gB$, and denote $\gA' = \alpha^{-1}(\gA\cap\gB')$.
Then the pair $(\gA, \gA\cap\gB')$ is a strong deformation retract of $(\gB,\gB')$ via the inclusion $\alpha$.
\end{enumerate}
\end{lemma}
\begin{proof}
The first statement is trivial.
Suppose $H\colon\gC\times\UInt\to\gC$ is a strong deformation retraction of $\gC$ into $e_{\gC}$, i.e.\ $H_0=\id_{\gC}$, $H_t(e_{\gC})=e_{\gC}$ for $t\in\UInt$, and $H_1(\gC) = \{e_{\gC}\}$.
Then the map
\begin{equation}\label{equ:deform_B_to_A}
    G\colon \gB\times\UInt\to\gB,
    \qquad
    G(b,t) = \smprA\bigl(H(\prC(b^{-1}),1-t))\bigr)\cdot b,
\end{equation}
is a strong deformation retraction of $\gB$ onto $\gA$, so $G_0=\id_{\gB}$, $G_t$ is fixed on $\gA$ for $t\in\UInt$, and $G_1(\gB) \subset \gA$.

Moreover, if $\smprA(\gC) \subset \gB'$, then~\eqref{equ:deform_B_to_A} shows that if $b\in\gB'$, then $G(b,t)\in\gB'$ as well.
In other words, $\gB'$ is invariant under $G$, and thus $G$ induces a strong deformation retraction of $\gB'$ onto $\sincl(\gA\cap\gB')$.
\end{proof}

\begin{lemma}\label{lm:monoids_homo}
Let $\smprM\colon\gA\to\gB$ be a homomorphism of monoids, so $\smprM(e_{\gA})=e_{\gB}$ and $\smprM(aa')=\smprM(a)\smprM(a')$ for all $a,a'\in\gA$.
Then for every invertible $a\in\gA$, its image $\smprM(a)$ is invertible in $\gB$.
\end{lemma}
\begin{proof}
We have that $e_{\gB} = \smprM(e_{\gA}) = \smprM(a a^{-1}) = \smprM(a)\smprM(a^{-1})$.
Therefore, $(\smprM(a))^{-1}=\smprM(a^{-1})$.
\end{proof}

\subsection*{Smooth even functions}
We will need the following statement.
\begin{lemma}[\rm H.~Whitney, \cite{Whitney:DJM_2:1943}]
\label{lm:Whitney_even_func}
Let $a>0$, $I_{a} = [-a;a]$, and $\gamma\in\Ci{I_a}{\bR}$ be an even function, that is $\gamma(-t)=\gamma(t)$ for all $t\in I_a$.
Then there exists a unique $\phi\in\Ci{[0;a]}{\bR}$ such that $\gamma(t) = \phi(t^2)$ for all $t\in\bR$.

Moreover, let $\CinftyEeven(I_a,\bR)$ be the space of even $\Cinfty$ functions on $I_a$.
Then the correspondence $\gamma\to\phi$ is an $\bR$-linear map $\delta:\CinftyEeven(I_a,\bR)\to\Ci{[0;a]}{\bR}$ being continuous between the $\Cinfty$ topologies.
\end{lemma}
\begin{proof}[Sketch of proof.]
Notice that $\phi\colon[0;a]\to\bR$ is uniquely defined by $\phi(t) = \gamma(\sqrt{\nrm{\az}})$, $t\in[0;a]$.
Such formula implies that $\phi$ is $\Cinfty$ only for $t\not=0$ and the main difficulty was to show that $\phi$ is in fact $\Cinfty$ near $0$ as well.
Smoothness of $\phi$ is proved by Whitney, and we need to derive the second statement about continuity of $\delta$.

Uniqueness of $\phi$ easily implies that $\delta$ is an $\bR$-linear map.
Hence, it suffices to verify continuity of $\delta$ at the zero function $0$ only.
One easily checks by induction that the identity $\gamma(t)=\phi(t^2)$ implies that for every for $r\geq1$ there exists some constant $A_r>0$ depending only on $r$ such that:
\[
    \sup\limits_{t\in[0;a]}
        \bigl\vert\tfrac{d^r\!\phi}{dt^i}(t)\bigr\vert
            \leq
        A_r \sum_{i=0}^{r+1} \sup\limits_{t\in I_a}
            \bigl\vert\tfrac{d^i\gamma}{dt^i}(t)\bigr\vert,
\]
This implies, that for each $r\geq0$ the map $\delta$ is continuous from $\Cr{r+1}$ topology of $\CinftyEeven(I_a,\bR)$ into $\Cr{r}$ topology of $\Ci{[0;a]}{\bR}$.
Hence, it is continuous between $\Cinfty$ topologies.

For example, notice that $\gamma'(t) = 2t \phi'(t^2)$.
Hence, $\gamma'(0)=0$, and thus, by Hadamard lemma, $\gamma'(t) = t \delta(t)$, where $\delta(t) = \smallint\limits_{0}^{1}\gamma''(st) ds$.
Therefore,
\[ \phi'(t^2) = \tfrac{1}{2}\delta(t) = \tfrac{1}{2}\smallint\limits_{0}^{1}\gamma''(st) ds,\]
and
\[
    \sup\limits_{t\in[0;a]} \nrm{\phi'(t)} \leq \tfrac{1}{2} \sup\limits_{t\in I_a} \nrm{\gamma''(t)}.
\]
We leave the other cases $r\geq2$ for the reader.
\end{proof}

\section{Stabilizers of functions under actions of diffeomorphism groups}
\label{sect:left-right-actions}

\subsection*{Left-right actions of diffeomorphism groups}
Let $\Mman$ be a smooth compact manifold.
Then the product $\DiffRM$ of groups of diffeomorphisms naturally acts from the \term{left} on the space of smooth functions $\CiMR$ by the following action map, see e.g.~\cite[Section~3]{MondNunoBallesteros:Sing:2020} for detailed discussions and references:
\begin{align*}
&\mu:\DiffRM \times \CiMR \to \CiMR, &
&\mu(\rphi,\dif,\func) = \rphi\circ\func\circ\dif^{-1}.
\end{align*}
It is usually referred as \term{left-right}.
Notice also that $\DiffM = \id_{\bR} \times \DiffM$ is a subgroup of $\DiffRM$, and thus we have an induced (still \term{left}) action
\begin{align*}
&\mu:\DiffM \times \CiMR \to \CiMR, &
&\mu(\dif,\func) = \func\circ\dif^{-1},
\end{align*}
which will be referred below as \term{right}%
\footnote{In fact, it will become a right action if we define it by $\mu(\dif,\func) = \func\circ\dif$.
However, it will be convenient to use the terms ``left-right'' and ``right'' to refer the sides at which we apply the corresponding diffeomorphisms to $\func$.}.
In terms of ``arrows'' these actions ``move down'' the horizontal arrow $f$:
\begin{align*}
&\xymatrix@R=1.5em{
  \Mman \ar[rr]^-{\func} \ar[d]_-{\dif}          && \bR \ar[d]^-{\rphi}    \\
  \Mman \ar[rr]^-{\rphi\circ\func\circ\dif^{-1}} && \bR
}&
&\xymatrix@R=1.5em{
  \Mman \ar[rr]^-{\func} \ar[d]_-{\dif} && \bR     \\
  \Mman \ar[rru]_-{\func\circ\dif^{-1}}
}
\end{align*}

Let $\func\in\CiMR$ and $\Xman\subset\Mman$ be a subset.
Denote by $\Diff(\Mman,\Xman)$ the subgroup of $\DiffM$ consisting of diffeomorphisms equal to the identity on $\Xman$.
Then one can define the stabilizers of $\func$ with respect to the above left-right and right actions:
\begin{align*}
    \StabLR{\func,\Xman} &:= \{(\rphi,\dif)\in\DiffR\times\DiffMX \mid  \rphi\circ\func\circ\dif^{-1} = \func \}, \\
    \StabR{\func,\Xman}  &:= \{\dif\in\DiffMX                     \mid  \func\circ\dif^{-1} = \func \}.
\end{align*}
If $\Xman$ is empty, then we will omit it from the notation.

Notice also that we have a canonical inclusion $\sincl\colon\StabR{\func,\Xman} \subset \StabLR{\func,\Xman}$, $\sincl(\dif) = (\id_{\bR},\dif)$, and therefore sometimes we will identify $\StabR{\func,\Xman}$ with its image $\{\id_{\bR}\} \times \StabR{\func,\Xman}$.

\begin{example}\rm
Define $\rphi,\func,\dif,q:\bR\to\bR$ by $\rphi(t)=4t$, $f(x) = x^2$, $\dif(x)=2x$, $q(x) = -x$.
Then the identities $4x^2 = (2x)^2$ and $(-x)^2=x^2$, mean respectively that $\rphi\circ\func = \func \circ \dif$ and $\func\circ q = \func$, so $(\rphi,\dif)\in\SMRf$ and $q\in\SMf$.
\end{example}

We will describe now the geometrical meaning of the above right- and left-right stabilizers in terms of level sets of $\func$.
Let $\AFoliation = \{\func^{-1}(t)\}_{t\in\bR}$ be the partition of $\Mman$ into the level sets of $\func$,  $\rphi\in\DiffR$, and $\dif\in\DiffM$.
Then it is easy to see that the following conditions are equivalent:
\begin{enumerate}[label={(R\arabic*)}, leftmargin=10ex]
    \item\label{enum:R1} $\func\circ\dif^{-1} = \func$, so $\dif\in\SMf$;
    \item\label{enum:R2} $\func = \func\circ\dif$;
    \item\label{enum:R3} $\dif(\func^{-1}(t))=\func^{-1}(t)$ for every $t\in\bR$, i.e.\ $\dif$ is an $\AFoliation$-leaf preserving diffeomorphism.
\end{enumerate}
Similarly, the following conditions are also equivalent:
\begin{enumerate}[label={(LR\arabic*)}, leftmargin=10ex]
    \item\label{enum:LR1} $\rphi\circ\func\circ\dif^{-1} = \func$, so $\dif\in\SMRf$;
    \item\label{enum:LR2} $\rphi\circ\func = \func\circ\dif$;
    \item\label{enum:LR3} $\dif(\func^{-1}(t))=\func^{-1}(\rphi(t))$ for all $t\in\bR$.
\end{enumerate}
In particular, condition~\ref{enum:LR3} implies that $\dif$ is an $\AFoliation$-foliated diffeomorphism.
However, if $\dif'$ is an $\AFoliation$-foliated diffeomorphism, then a priori we can not claim that there exists $\rphi\in\DiffR$ such that $(\rphi,\dif')\in\SMRf$.

Let us clarify the relations between stabilizers and foliated diffeomorphisms.
\begin{lemma}\label{lm:SMRf_image}
Let $\func\in\Ci{\Mman}{\bR}$ and $\Aman:= \func(\Mman) \subset \bR$ be its image.
Then for every $(\rphi,\dif)\in\SMRf$, we have that $\rphi(\Aman)=\Aman$.
Moreover, if $\psi\in\DiffR$ is another diffeomorphism such that $\psi = \rphi$ on $\Aman$, then $(\psi',\dif)\in\SMRf$.
\end{lemma}
\begin{proof}
Let $a\in \Aman$, so $a=\func(\px)$ for some $\px\in\Mman$.
Then $\rphi(a) = \rphi\circ\func(\px) = \func\circ\dif(\px) \in \Aman$, so $\rphi(\Aman)\subset \Aman$.
Applying the same arguments to the inverse $(\rphi^{-1},\dif^{-1})=(\rphi,\dif)^{-1}\in\SMRf$, we obtain that $\rphi^{-1}(\Aman)\subset \Aman$, whence $\rphi(\Aman)=\Aman$.
Moreover, if $\psi = \rphi$ on $\Aman$, then for each $\px\in\Mman$ we have that $\psi\circ\func(\px)=\rphi\circ\func(\px)=\func\circ\dif(\px)$.
Thus, $(\psi,\dif)\in\SMRf$ as well.
\end{proof}

Let $\func\in\Ci{\Mman}{\bR}$ and $\Xman \subset \Mman$ be a subset.
Then its image $\Aman:= \func(\Mman) \subset \bR$ is a finite union of closed intervals and points, and therefore a submanifold of $\bR$.
Let $\DiffA$ be the group of diffeomorphisms of $\Aman$.
In view of Lemma~\ref{lm:SMRf_image}, it is more natural to regard $\func$ as a surjective function $\func\in\Ci{\Mman}{\Aman}$, and study the stabilizer of $\func$ with respect to the corresponding ``left-right'' action of $\DiffAM$.
Therefore, it is more convenient to consider instead of $\SMRf$ the following groups:
\begin{align*}
\StabLRImgPl{\func,\Xman} &:= \{ (\rphi,\dif) \in \DiffPlA \times \DiffMX \mid \rphi\circ\func=\func\circ\dif \}, \\
\StabLRImg{\func,\Xman}  &:= \{ (\rphi,\dif) \in \DiffA  \times \DiffMX \mid \rphi\circ\func=\func\circ\dif \}.
\end{align*}
Then $\SMf \equiv \id_{\Aman}\times\SMf$ is a subgroup of $\SPlMAf \subset \SMAf$, and the properties~\ref{enum:LR1}-\ref{enum:LR3} still hold for $\SMAf$ instead of $\SMRf$.
Moreover, we have the following lemma.
\begin{lemma}\label{lm:DAM_DM}
Let $\prA\colon\DiffAM\to\DiffA$ and $\prM\colon\DiffAM\to\DiffM$ be natural projections, so $\prA(\rphi,\dif) = \rphi$ and $\prM(\rphi,\dif) = \dif$.
Then the following statements hold.
\begin{enumerate}
\item\label{enum:SMAf:homo} $\prA$ and $\prM$ are homomorphisms;
\item\label{enum:SMAf:Smf_DTlp} $\SMf = \DiffAFolLp$;
\item\label{enum:SMAf:ker_prA} $\ker(\restr{\prA}{\SMAf}) = \id_{\Aman} \times \SMf$;
\item\label{enum:SMAf:img_prM} $\prM(\SMAf) \subset \DiffAFol$;
\item\label{enum:SMAf:ker_prM} if $(\rphi,\dif),(\rphi',\dif)\in\SMRf$, then $\rphi=\rphi'$, i.e.\ the map $\restr{\prM}{\SMAf}\colon \SMAf \to \DiffM$ is injective.
\end{enumerate}
\end{lemma}
\begin{proof}
Statements~\ref{enum:SMAf:homo} and~\ref{enum:SMAf:ker_prA} are trivial, \ref{enum:SMAf:Smf_DTlp} coincides with~\ref{enum:R3}, while~\ref{enum:SMAf:img_prM} coincides with~\ref{enum:LR3}.

\ref{enum:SMAf:ker_prM}
Notice that the kernel of $\restr{\prM}{\SMAf}$ consists of pairs of the form $(\rphi,\id_{\Mman})\in\SMAf$ satisfying $\rphi\circ\func = \func$, whence $\rphi\in\DiffA$ is fixed on the image $\Aman$ of $\func$, so $\rphi =\id_{\Aman}$.
\end{proof}
In particular, we get the following commutative diagram in which the upper row is exact at the two middle items (though $\prM$ is not necessarily surjective):
\begin{equation}\label{equ:diag_rel_Stabs}
\begin{gathered}
\xymatrix@R=1.5em{
1 \ar[r] &   \SMf \ar[rr]^-{\sincl\colon \dif\,\mapsto\,(\id_{\Aman},\dif)}  &&
    \SMAf \ar[r]^-{\prA}  \ar@{^(->}[d]^-{\prM} & \DiffA \\
         &   \DiffAFolLp   \ar@{=}[u] \ar@{^(->}[rr]                      &&
    \DiffAFol
}
\end{gathered}
\end{equation}

\subsection*{Condition $\JProp$}
In~\cite{Maksymenko:BSM:2006} the author gave wide conditions on $\func$ under which the inclusion $\sincl\colon\SMf \subset \SPlMAf$ is a homotopy equivalence, see Theorem~\ref{th:SMf_SMRf_he} below.
We will briefly recall that result.

\begin{definition}\label{def:propJ}
Let $\func\colon\bR^n\to\bR$ be a $\Cinfty$ function.
Say that $\func$ \term{has property \JProp\ at $\pu\in\bR^{n}$} if there exists a neighborhood $\Uman$ of $\pu$ and $\Cinfty$ functions $\alpha_1,\ldots,\alpha_n\colon\Uman\to\bR$ such that
$\func(\px) - \func(\pu) = \sum\limits_{i=1}^{n} \func'_{\pxi{i}}(\px) \, \alpha_i(\px)$, $\px\in\Uman$.
\end{definition}

\begin{example}\rm
Let $\func\colon\bR^{n}\to\bR$ be a $\Cinfty$ homogeneous function of order $k$, that is $\func(t\px)=t^{k}\func(\px)$ for all $t\geq0$ and $\px\in\bR^{n}$.
Then, by the well known Euler identity
\[
    \func = \func'_{x_1}\tfrac{x_1}{k} + \cdots + \func'_{x_n}\tfrac{x_n}{k},
\]
so $\func$ has property \JProp\ at the origin $0\in\bR^{n}$.
\end{example}

Equivalently, denote by $\Calg{\pu}{\bR^n}$ the algebra of germs at $\pu$ of $\Cinfty$ functions $\bR^n\to\bR$, and for each $\func\in\Calg{\pu}{\bR^n}$ let $\Jideal{\pu}{\func}$ be the ideal in $\Calg{\pu}{\bR^n}$ generated by partial derivatives of $\func$.
Then property \JProp\ means that the germ at $\pu$ of the function $\gfunc(x) = \func(x)-\func(\pu)$ belongs to $\Jideal{\pu}{\func}$.
The following simple lemma shows that the property \JProp\ does not depend on local coordinates at $\pu$, and so it is well-defined for functions on manifolds.
\begin{lemma}\label{lm:Jrop_local_coords}
Let $\dif=(\dif_1,\ldots,\dif_n)\colon(\bR^{m},\pv)\to(\bR^{n},\pu)$ be a germ of a $\Cinfty$ map, and $\dif^{*}\colon\Calg{\pu}{\bR^n} \to \Calg{\pv}{\bR^m}$, $\dif^{*}(\func) = \func\circ\dif$, be the induced algebra homomorphism.
Then
\[
    \Jideal{\pv}{\dif^{*}(\func)} \subset \dif^{*}\bigl( \Jideal{\pu}{\func} \bigr).
\]
In particular, if $\dif$ is a diffeomorphism, then $\func\in\Jideal{\pu}{\func}$ iff $\dif^{*}(\func) \in \Jideal{\pu}{\dif^{*}(\func)}$.
\end{lemma}
\begin{proof}
Let $\py=(\pyi{1},\ldots,\pyi{m})\in\bR^{m}$.
Then for each $k=1,\ldots,m$ we have that
\[
    \ddd{(\dif^{*}(\func))}{y_k}(\py)
        = \ddd{(\func\circ\dif)}{y_k}(\py)
        = \sum_{i=1}^{n} \ddd{\func}{x_i}(\dif(\py)) \ \ddd{\dif_i}{\py}(\py)
        = \sum_{i=1}^{n} \dif^{*}(\func'_{x_i})(\py) \ \ddd{\dif_i}{\py}(\py),
\]
i.e.\ partial derivatives of $\dif^{*}(\func)$ are linear combinations with smooth coefficients of the images of partial derivatives $\dif^{*}(\func'_{x_i})$ of $\func$.
Hence, $\Jideal{\pv}{\dif^{*}(\func)} \subset \dif^{*}\bigl( \Jideal{\pu}{\func} \bigr).$
\end{proof}
\begin{theorem}[{\rm\cite[Theorem~1.3]{Maksymenko:BSM:2006}}]
\label{th:SMf_SMRf_he}
Let $\Mman$ be a smooth connected compact manifold, $\func\in\CiMR$, and $\Aman=\func(\Mman)$ be its image.
Suppose that
\begin{enumerate}[label={\rm(\alph*)}]
\item\label{enum:func:fin_crval_bd}
$\func$ takes a constant value at each connected component of $\partial\Mman$ and has only finitely many critical values;
\item\label{enum:func:prop_J}
$\func$ has property $\JProp$, see Definition~\ref{def:propJ}, at every critical point $x\in\Mman$.
\end{enumerate}
Then $\id_{\Aman}\times\SMf$ is a strong deformation retract of $\SPlMAf$.
\end{theorem}
\begin{proof}[Remarks to the proof.]
First note that if $\func$ is constant, then $\Aman$ is a point, and conditions~\ref{enum:func:fin_crval_bd} and~\ref{enum:func:prop_J} are satisfied.
Moreover, in this case $\SMf = \Diff(\Mman)$, and $\SPlMAf = \id_{\Aman} \times \Diff(\Mman) = \id_{\Aman} \times \SMf$, so the statement of the theorem is trivial.

Thus assume that $\func$ has at least two critical points.
Let $\Qman = \{ a_0, a_1, \cdots, a_n\} \subset \Aman$ be all the values of $\func$ at critical points and boundary components of $\Mman$.
Condition~\ref{enum:func:fin_crval_bd} implies that $\Qman$ is finite.
Let also $\Diff(\Aman,\Qman)$ be the subgroup of $\DiffA$ fixed on $\Qman$.
One easily checks that if $(\phi,\dif)\in\SPlMAf$, then $\phi\in\Diff(\Aman,\Qman)$, i.e.\ $\prA(\SPlMAf)\subset\Diff(\Aman,\Qman)$.
Moreover, each $\phi\in\Diff(\Aman,\Qman)$ preserves orientation of $\Aman$, and $\Diff(\Aman,\Qman)$ is convex in $\Ci{\Aman}{\Aman}$, and therefore contractible.

The main technical result of~\cite[Theorem~1.3]{Maksymenko:BSM:2006} is based on condition~\ref{enum:func:prop_J} and claims that there exists a continuous homomorphism $\theta\colon\Diff(\Aman,\Qman)\to\SPlMAf$ such that $\phi\circ\func=\func\circ\theta(\phi)$.
In other words, $\theta$ is a section of $\prA\colon\SPlMAf\to\Diff(\Aman,\Qman)$, so the statement of this theorem  follows from Lemma~\ref{lm:sect_of_homo}.
\end{proof}

\section{Fiberwise homogeneous functions on vector bundles}
\label{sect:fib_hom_func_vb}
In this section we prove Theorem~\ref{th:def_2_hom_func} including Theorem~\ref{th:Dlp_Dfol_he} as a particular case.
\subsection*{Smooth homogeneous functions}
Recall that a continuous function $\func\colon\bR^{n}\to\bR$ is \term{homogeneous} of some degree $k\geq0$, whenever $\func(t\pv)=t^k\func(\pv)$ for all $t>0$ and $\pv\in\bR^{n}$.
The following simple lemma seems to be a classical result.
However, the author did not find precise references, though there are discussions on this question in internet resources, e.g.~\cite{Dubovski:MOF:2017}.
\begin{lemma}\label{lm:homog_func}
Let $\func\colon\bR^{n}\to\bR$ be a homogeneous $\Cr{k}$ function of integer degree $k\geq0$.
Then $\func$ is a homogeneous polynomial of degree $k$.
\end{lemma}
\begin{proof}
If $k=0$, then the assumption that $\func$ is continuous and homogeneous of degree $0$, means that $\func(t\pv)=t^0\func(\pv)=\func(\pv)$ for all $t\geq0$ and $\pv\in\bR^{n}$.
In particular, for $t=0$ we get that $\func(\pv)=\func(0)$ for all $\pv\in\bR^{n}$, so $\func$ is constant.

For $k>0$, applying $\ddd{}{\pvi{i}}$, $i=1,\ldots,n$, to both sides of the identity $\func(t\pv)=t^k\func(\pv)$, we get $t\func'_{\pvi{i}}(t\pv) = t^k\func'_{\pvi{i}}(t\pv)$, whence $\func'_{\pvi{i}}(t\pv) = t^{k-1}\func'_{\pvi{i}}(t\pv)$, i.e.\ every partial derivative $\func'_{\pvi{i}}$ is a homogeneous $C^{k-1}$ function.
Then, by induction on $k$, $\func'_{\pvi{i}}$ must be a homogeneous polynomial of degree $k-1$.
Hence, by Euler's identity, $\func(\pv) = \tfrac{1}{k}\sum_{i=1}^{n} \pvi{i} \func'_{\pvi{i}}(\pv)$ is a homogeneous polynomial of degree $k$.
\end{proof}

We will identify $\Bman$ with the image of the zero section of $\vbp$ in $\Eman$.
A continuous function $\func\colon\Eman\to\bR$ is called \term{homogeneous} of degree $k\geq0$ (or \term{$k$-homogeneous}) whenever $\func(t \px) = t^{k}\func(\px)$ for all $t\geq0$ and $\px\in\Eman$.

More generally, let $\vbp\colon\Eman\to\Bman$ be a smooth vector bundle of rank $n$ over a manifold $\Bman$.
Assume that $\func\colon\Eman\to\bR$ is homogeneous.
Then $\Bman\subset\func^{-1}(0)$, however, in general, this inclusion is non-strict.
We will say that $\func$ is \term{definite}, whenever $\func(\px)>0$ for all $\px\in\Eman\setminus\Bman$.
In particular, $\Bman = \func^{-1}(0)$.

\begin{corollary}\label{cor:homog_func}
Let $\vbp\colon\Eman=\Bman\times\bR^{n} \to \Bman$ be a trivial vector bundle, and $\func\colon\Eman\to\bR$ be a $k$-homogeneous $\Cr{r}$ function with $k\leq r$.
Then
\begin{equation}\label{equ:f_homog_on_vb}
    \func(\py,\pvi{1},\ldots,\pvi{n})
    =
    \sum\limits_{
            \substack{
                    i_1,\ldots,i_n\in\{0,\ldots,k\} \\
                    i_1+\cdots+i_n = k
            }
    }
    a_{i_1,\ldots,i_n}(\py)\, \pvi{1}^{i_1} \cdots \pvi{n}^{i_n},
\end{equation}
where each $a_{i_1,\ldots,i_n}\colon\Bman\to\bR$ is some $\Cr{r-k}$ function.
\end{corollary}
\begin{proof}
Notice that $k$-homogeneity of $\func$ means that for every $(\py,\pu)\in\Bman\times\bR^{n}$ and $t\geq0$ we have that $\func(\py,t\pu) = t^{k}\func(\py,\pu)$.
Now the proof follows from the Lemma~\ref{lm:homog_func}.
\end{proof}
\begin{corollary}\label{cor:homog_func_prop_J}
Let $\func\colon\Eman\to\bR$ be a $\Cinfty$ homogeneous function of degree $k\geq2$.
Then $\func$ satisfies condition \JProp\ at each $\px\in\Bman$.
\end{corollary}
\begin{proof}
Due to Lemma~\ref{lm:Jrop_local_coords} one can pass to a local trivialization of $\vbp$ at $\px$, and thus assume that $\vbp\colon\Eman\to\Bman$ is a trivial vector bundle.
Since $k\geq2$, every point of $\Bman$ is critical.
Then by~\eqref{equ:f_homog_on_vb} and ``Euler's identity with respect to the coordinates $\pvi{1},\ldots,\pvi{n}$'' we have that $\func(\py,\pv) = \tfrac{1}{k}\sum_{i=1}^{n} \pvi{i}\func'_{\pvi{i}}(\py,\pv)$.
Hence, $\func$ satisfies property \JProp\ at $\px$.
\end{proof}

\subsection*{Stabilizers of homogeneous functions}
Let $\vbp\colon\Eman\to\Bman$ be a smooth vector bundle of rank $n$ over a manifold $\Bman$.
Let also $\func\colon\Eman\to[0;+\infty)$ be a $\Cinfty$ function such that $1$ is a regular value of $\func$, so $\ATor = \func^{-1}(\UInt)$ is a submanifold of $\Eman$, and $\AFoliation = \{\func^{-1}(t) \mid t\in\UInt\}$ be the partition of $\ATor$ into level sets of $\func$.

The following statement is a particular case of~\cite[Theorem~1.3]{Maksymenko:BSM:2006}.
However, the proof is essentially simpler, and it additionally takes to account the behavior of diffeomorphisms on $\partial\ATor$.

\begin{lemma}[{\rm c.f.~\cite[Theorem~1.3]{Maksymenko:BSM:2006}}]
\label{lm:sigma:k-homog}
Let $k\geq1$, $\func\colon\Eman\to\bR$ be a $k$-homoge\-neous $\Cinfty$ function, $\ATor = \func^{-1}(\UInt)$, and $\AFoliation = \{\func^{-1}(t) \mid t\in\UInt\}$ be the partition of $\ATor$ into level sets of $\func$.
Then there exists a homomorphism $\smprA\colon\DiffPlI\to\FolDiffdATor$ such that $\phi\circ\func = \func\circ\smprA(\dif)$ for all $\phi\in\DiffPlI$.
In particular, $(\phi,\smprA(\phi))\in\SMAf$.

If $\ATor$ is compact, then $\smprA$ is continuous, and the pair $\bigl(\SMf,\ \SMfd\bigr)$ is a strong deformation retract of the pair $\bigl(\SMAf,\ \SMAfd\bigr)$ with respect to the natural inclusion $\sincl\colon\SMf\equiv\{\id_{\UInt}\}\times\SMf\subset\SMAf$.
\end{lemma}
\begin{proof}
Let $\phi\in\DiffPlI$, so $\phi\colon\UInt\to\UInt$ is a $\Cinfty$ function such that $\phi(0)=0$, $\phi(1)=1$, and $\phi'>0$.
Then by the arguments of Hadamard lemma:
\[
    \phi(t) = \smallint\limits_0^t \phi'(u)du =
    \left\vert
        \substack{\text{replace $u = st$,} \\
                    \text{then $du = s\,dt$}}
    \right\vert =
    t \underbrace{\smallint\limits_0^1 \phi'(st)dt}_{\gfunc_{\phi}(t)} = t \gfunc_{\phi}(t),
\]
where $\gfunc_{\phi}\colon\UInt\to\bR$ is $\Cinfty$.
Moreover, $\gfunc_{\phi}(t) > 0$ for all $t\in\UInt$ and $\gfunc_{\phi}(0)=\phi'(0)$.

\newcommand\kpow[1]{\left[#1\right]^{1/k}}

Define the map
\begin{align*}
&\smprA\colon\DiffPlI \to \Ci{\ATor}{\ATor},&
&\smprA(\phi)(\px) = \kpow{\gfunc_{\phi}(\func(\px))}\px.
\end{align*}
We claim that the image of $\smprA$ is contained in $\DiffAFol$, and the induced map $\smprA\colon\DiffPlI \to \DiffAFol$ is the desired homomorphism.

\begin{enumerate}[wide, label={\arabic*)}]
\item\label{enum:theta_gen:id_to_id}
First we show that $\smprA$ is a homomorphism of \term{monoids} with respect to the natural composition of maps.
Then by Lemma~\ref{lm:monoids_homo} it sends invertible elements to invertible, and therefore $\smprA(\dif)$ will be a diffeomorphism of $\ATor$.

a) Indeed, if $\phi=\id_{\UInt}$, then $\gfunc_{\phi}(t)\equiv 1$, whence $\smprA(\id_{\UInt}) = \id_{\ATor}$.

b) Furthermore, let $\iii{\phi},\jjj{\phi}\in\DiffPlI$ be two diffeomorphisms.
Then $\iii{\phi}(t) = t \iii{\gfunc}(t)$ and $\jjj{\phi}(t) = t\jjj{\gfunc}(t)$ for unique $\Cinfty$ functions $\iii{\gfunc},\jjj{\gfunc}\colon\UInt\to\bR$ such that
\begin{align*}
\smprA(\iii{\phi})(\px) &= \kpow{\iii{\gfunc}(\func(\px))}\px, &
\smprA(\jjj{\phi})(\px) &= \kpow{\jjj{\gfunc}(\func(\px))}\px.
\end{align*}
Hence
\begin{align*}
    \smprA(\jjj{\phi})\circ\smprA(\iii{\phi})(\px) &=
    \smprA(\jjj{\phi})\bigl(\kpow{\iii{\gfunc}(\func(\px))}\,\px\bigr) \\
&=
    \kpow{
        \jjj{\gfunc}\Bigl(
            \func\bigl(
                    \kpow{\iii{\gfunc}(\func(\px))}\,\px
                \bigr)
            \Bigr)
        }
        \cdot
        \kpow{\iii{\gfunc}(\func(\px))}\,\px = \\
    &=
    \kpow{
        \jjj{\gfunc}\bigl[\iii{\gfunc}(\func(\px)) \cdot\func(\px)\bigr]
        }
        \cdot
        \kpow{\iii{\gfunc}(\func(\px))}\,\px.
\end{align*}
On the other hand, $\jjj{\phi}\circ\iii{\phi}(t) = \iii{\phi}(t) \cdot \jjj{\gfunc}(\iii{\phi}(t)) = t \cdot \underbrace{\iii{\gfunc}(t) \cdot\jjj{\gfunc}(t \iii{\gfunc}(t))}_{\bar{\gfunc}}$, whence
\begin{align*}
    \smprA\bigl(\jjj{\phi}\circ\iii{\phi}\bigr)(\px) &=
    \kpow{\bar{\gfunc}(\func(\px))}\,\px \\
    &=
    \kpow{
        \iii{\gfunc}(\func(\px)) \cdot
        \jjj{\gfunc}\bigl[\func(\px) \cdot\iii{\gfunc}(\func(\px))\bigr]
    }\,\px =
    \smprA(\jjj{\phi})\circ\smprA(\iii{\phi})(\px).
\end{align*}

\item
Let us verify that $(\phi,\smprA(\phi))\in\SMRf$.
Indeed,
\[
    \func\circ\smprA(\phi)(\px)
        = \func\Bigl( \kpow{\gfunc_{\phi}\circ\func(\px)}\,\px  \Bigr)
        = (\gfunc_{\phi}\circ\func(\px)) \cdot \func(\px)
        = \phi\circ \func(\px).
\]
This implies that $(\phi,\smprA(\phi))\in\SMRf$, and in particular, $\smprA(\phi)\in\DiffAFol$.

\item
Finally, note that $\gfunc_{\phi}(1) = \phi(1)/1 = 1$, whence if $\px\in\partial\ATor=\func^{-1}(1)$, then $\smprA(\phi)(\px)=\px$, so $\smprA(\phi)$ is fixed on $\partial\ATor$.
Thus, $\smprA(\phi)\in\FolDiffdATor$.

\item
Suppose $\ATor$ is compact.
Then continuity of $\smprA$ directly follows from the formulas for $\smprA(\dif)$.

Consider the following short exact sequence, see~\eqref{equ:diag_rel_Stabs}:
\[
    1 \to \SMf \xrightarrow{\sincl} \SMAf \xrightarrow{\prM}  \DiffPlI \to 1.
\]
Then the map $\widehat{\smprA}:\DiffPlI\to\SMAf$, $\widehat{\smprA}(\phi)=(\phi,\smprA(\phi))$, is a continuous section of $\prA$ and its image is contained in $\SMAfd$.
Moreover, $\DiffPlI$ is contractible into $\id_{\UInt}$ via the homotopy
\begin{align*}
   &H\colon\DiffPlI\times\UInt\to\DiffPlI, &
   &H(\phi,t) = (1-t)\phi + t \id_{\UInt}.
\end{align*}
Then, due to Lemma~\ref{lm:sect_of_homo}, the pair $(\SMf,\SMfd)$ is a strong deformation retract of $(\SMAf,\SMAfd)$ with respect to the inclusion map $\sincl$.
\qedhere
\end{enumerate}
\end{proof}

\subsection*{Foliated maps for homogeneous functions}
Let $\vbp_i\colon\Eman_i\to\Bman_i$ for $i=0,1$ be a smooth vector bundle of some rank $n_i$ over a compact manifold $\Bman_i$ and $\func_i\colon\Eman_i\to[0;+\infty)$ be a $\Cinfty$ function such that $1\in\bR$ is a regular value for $\func$.
Then $\ATor_i := \func_i^{-1}(\UInt)$ is a submanifold of $\Eman_i$.
Let also
\[
    \AFoliation_i = \{\func_i^{-1}(t) \mid t\in\UInt\}
\]
be the partition of $\ATor_i$ into the level sets of $\func$.

A map $\dif\colon\ATor_0 \to \ATor_1$ will be called \term{$(\AFoliation_0,\AFoliation_1)$-foliated} if for each leaf $\omega$ of $\AFoliation_0$ its image $\dif(\omega)$ is contained in some leaf of $\AFoliation_1$.
Denote by $\CiTwoFol$ the subset of $\Ci{\ATor_0}{\ATor_1}$ consisting of $(\AFoliation_0,\AFoliation_1)$-foliated maps.
If $\ATor_0$ and $\ATor_1$ are diffeomorphic, then we denote by $\DiffTwoFol$ the \term{set} of all $(\AFoliation_0,\AFoliation_1)$-foliated diffeomorphisms.

\begin{lemma}\label{lm:F0F1_folated}
Suppose the functions $\func_0$ and $\func_1$ are $2$-homogeneous and definite.
\begin{enumerate}[wide]
\item\label{enum:sigma:exists}
Then for each $(\AFoliation_0,\AFoliation_1)$-foliated map $\dif\colon\ATor_0 \to \ATor_1$ there exists a unique function $\smprM(\dif)\colon\UInt\to\UInt$ such that $\smprM(\dif)\circ\func_0 = \func_1\circ\dif$.
Moreover, if $\dif$ is $\Cinfty$, then $\smprM(\dif)$ is also $\Cinfty$ and the correspondence $\dif\mapsto\smprM(\dif)$ is a continuous map $\smprM\colon\CiTwoFol\to\Ci{\UInt}{\UInt}$.

\item\label{enum:sigma:dif_to_dif}
If $\ATor_0$ and $\ATor_1$ are diffeomorphic, then $\smprM(\dif)\in\DiffPlI$ is a preserving orientation diffeomorphism of $\UInt$ for each $\dif\in\DiffTwoFol$.

\item\label{enum:sigma:homomorphism}
Suppose $\vbp_0=\vbp_1$ is the same vector bundle and $\func_0=\func_1\colon \Eman_0=\Eman_1 \to \bR$, so $\ATor_0=\ATor_1$, $\AFoliation_0=\AFoliation_1$, and thus $\DiffTwoFol=\FolDiff(\AFoliation_0)$.
Then $\smprM\colon\Ci{\AFoliation_0}{\AFoliation_0}\to\Ci{\UInt}{\UInt}$ is a (continuous) homomorphism of monoids.
In particular, $\smprM(\FolDiff(\AFoliation_0)) \subset \DiffPlI$.
\end{enumerate}
\end{lemma}
\begin{proof}
\ref{enum:sigma:exists}
The assumption that $\dif\colon\ATor_0 \to \ATor_1$ is $(\AFoliation_0,\AFoliation_1)$-foliated means that for each leaf $A_t = \func_0^{-1}(t)$, $t\in\UInt$ of $\AFoliation_0$ there exist a unique leaf $B_{t'} = \func_1^{-1}(t')$ such that $\dif(A_t)\subset B_{t'}$.
Define the function $\rphi:\UInt\to\UInt$ by $\rphi(t)=t'$.
Then $\rphi\circ\func_0=\func_1\circ\dif$, and then we put $\smprM(\dif)=\rphi$.
Hence, such $\smprM(\dif)$ is unique.

Suppose $\dif$ is $\Cinfty$.
To prove that $\smprM(\dif)$ is $\Cinfty$ as well we will use Whitney Lemma~\ref{lm:Whitney_even_func}.
Fix any point $\pw\in\partial\ATor_0$, so $\func_0(\pw)=1$, and consider the path $\eta\colon[-1;1]\to\ATor_0$ given by $\eta(t) = t\pw$.
Then
\begin{equation}\label{equ:f_0_eta__t2}
    \func_0\circ\eta(t) = \func_0(t\pw) = t^{2}\func_0(\pw) = t^{2}.
\end{equation}
In particular, for each $t\in[-1;1]$ the points $\eta(t)$ and $\eta(-t)$ belong to the same leaf $\func_0^{-1}(t^{2})$ of $\AFoliation_0$.

Define the following $\Cinfty$ function $\gamma = \func_1\circ\dif\circ\eta\colon[-1,1]\to\bR$.
Since $\dif$ sends level sets of $\func_0$ to level sets of $\func_1$, the points $\dif(\eta(t))$ and $\dif(\eta(-t))$ also belong to the same level set of $\func_1$, that is
\[
    \gamma(-t) = \func_1\circ\dif\circ\eta(-t) = \func_1\circ\dif\circ\eta(t) = \gamma(t).
\]
Hence, $\gamma$ is an even function.
Then, due to Lemma~\ref{lm:Whitney_even_func}, there exists a unique $\Cinfty$ function $\rphi:\UInt\to\UInt$ such that $\gamma(t) = \rphi(t^2)$.
Thus,
\begin{equation}\label{equ:fh_phif}
    \func_1\circ\dif\circ\eta(t) = \gamma(t) = \rphi(t^2) \stackrel{\eqref{equ:f_0_eta__t2}}{=} \rphi\circ\func_0\circ\eta(t).
\end{equation}
We claim that $\func_1\circ\dif\equiv \rphi\circ\func_0$ on all of $\ATor_0$.
Indeed, let $\py\in\ATor_0$ and $\func_0(\py) = s$ for some $s\in\UInt$.
Then $\func_0(\eta(\sqrt{s})) = s$ as well, so the points $\py$ and $\eta(\sqrt{s})$ belong to the same level set $\func_0^{-1}(s)$ of $\func_0$.
Hence, $\dif(\py)$ and $\dif(\eta(\sqrt{s}))$ also belong to the same level set of $\func_1$.
Therefore,
\begin{equation}\label{equ:fh_phif_full}
    \func_1\circ\dif(\py)
        = \func_1\circ\dif\circ\eta(\sqrt{s})
        \stackrel{\eqref{equ:fh_phif}}{=} \rphi\circ\func_0\circ\eta(\sqrt{s})
        = \rphi\circ\func_0(\py).
\end{equation}

Continuity of the correspondence $\dif\mapsto\rphi$ also follows from Lemma~\ref{lm:Whitney_even_func}.

\ref{enum:sigma:dif_to_dif}
Suppose $\dif$ is a diffeomorphism.
We should show that then $\rphi$ is a preserving orientation diffeomorphism of $\UInt$.
It suffices to check the following two properties of $\gamma$.
\begin{enumerate}[label={\rm(\alph*)}]
\item\label{enum:gamma_prop:t>0} $\gamma'(t)>0$ for $t>0$;
\item\label{enum:gamma_prop:t=0} $\gamma(t) = t^2 \delta(t)$ for some $\Cinfty$ function $\delta:\UInt\to\bR$ such that $\delta(0)>0$.
\end{enumerate}
Assuming they are proved, let us show that $\rphi\in\DiffPlI$.
Indeed, it is evident that $\rphi(0)=0$ and $\rphi(1)=1$, so we need to verify that $\rphi'>0$.
Since $\rphi(s) = \gamma(\sqrt{s})$ for $s\in\UInt$, we get from~\ref{enum:gamma_prop:t>0} that $\rphi'(s)>0$ for $s\in(0;1]$.
Furthermore, as $\rphi(0)=0$, we have by the Hadamard lemma that $\rphi(s) = s \psi(s)$ for some $\Cinfty$ function $\psi\colon\UInt\to\bR$ such that $\psi(0)=\rphi'(0)$.
Then, due to~\ref{enum:gamma_prop:t=0}, $t^2\delta(t) = \gamma(t) = \rphi(t^2) = t^2 \psi(t^2)$, whence $\delta(t)=\psi(t^2)$, and therefore $\rphi'(0) = \psi(0) = \delta(0)>0$.

{\bfseries Proof of \ref{enum:gamma_prop:t>0}.}
Note that if $\zeta\colon[a,b]\to\ATor_0$ is a smooth path not passing through $\Bman$, then for each $t\in[a;b]$ the following conditions are equivalent:
\begin{itemize}
\item $t$ is a regular point of the function $\func_0\circ\zeta\colon[a;b]\to\bR$;
\item $\zeta$ is transversal at $t$ to the level set $\func_0^{-1}(\func_0\circ\zeta(t))$.
\end{itemize}
Since $\func_0\circ\eta(t)=t^2$, it follows that the function $\func_0\circ\eta\colon[-1,1]\to\bR$ has a unique critical point $t=0$, and thus $\eta$ is transversal to all level sets $\func_0^{-1}(s)$ for all $s\in(0;1]$.

On the other hand, as $\dif$ sends level sets of $\func_0$ to level sets of $\func_1$, the path $\dif\circ\eta$ must also be transversal to the level sets of $\func_1$ for all $t\not=0$.
More precisely, if $\eta$ is transversal at some $t\not=0$ to the leaf $\Lman = \func_0^{-1}(t^2)$, then $\dif\circ\eta$ is transversal at $t$ to the leaf
\[ \dif(\Lman) = \func_1^{-1}(\func_1\circ\dif\circ\eta(t)) = \func_1^{-1}(\gamma(t)).\]
This implies that \term{the function $\gamma=\func_1\circ\dif\circ\eta$ has a unique critical point $t=0$}.

Moreover, since $\gamma(-1)=\gamma(1)=1$ and $\gamma(0)=0$, we see that $\gamma$ decreases on $[-1,0)$ and increases on $(0;1]$.
In particular, $\gamma'(t)<0$ for $t<0$ and $\gamma'(t)>0$ for $t>0$.

{\bfseries Proof of~\ref{enum:gamma_prop:t=0}.}
Let $q_0:\Uman\times\bR^{n}\to\Uman$ and $q_1:\Vman\times\bR^{n}\to\Vman$ be vector bundle trivializations of $\vbp$ over open neighborhood $\Uman$ of $\eta(0)$ and $\Vman$ of $\dif(\eta(0))$ respectively.
Then $\dif$ has local representation as an embedding $\dif=(\dif_0,\dif_1,\ldots,\dif_n)\colon\Uman\times\bR^{n} \supset \Wman \to \Vman\times\bR^{n}$ of some open neighborhood $\Wman$ of $\eta(0)$ in $\Uman\times\bR^{n}$, where $\dif_0:\Wman\to\Vman$ is a $\Cinfty$ map, and each $\dif_i:\Wman\to\bR$ is a $\Cinfty$ function.

One can assume that the image of $\eta$ is contained in $\Wman$, and $\pw = (\px, \bar{\pu}) \in \Uman\times\bR^{n}$ are the coordinates of $\pw$, so $\eta(t) = (\px, t\bar{\pu})$.

Then by Corollary~\ref{cor:homog_func} the restriction of $\func_1$ to each fiber of $\vbp$ is a $2$-homogeneous polynomial, i.e.\
$\func_1(\py,\pv) =\sum\limits_{1\leq i,j \leq n} a_{ij}(\py) \pvi{i} \pvi{j}$ for all $(\py,\pv)\in\Vman\times\bR^{n}$.
One can assume that $a_{ij}=a_{ji}$, so we get a symmetric matrix $A(\py)=(a_{ij})$, such that $\func_1(\py,\pv) = \pv A(\py) \pv^{t}$, where $\pv=(\pvi{1},\ldots,\pvi{n})$, and $\pv^{t}$ is the transposed vector column.
Hence,
\[
    \gamma(t) = \func_1\circ\dif\circ\eta(t)
                = \func_1\bigl(\dif(\px,t\bar{\pu}) \bigr)
                = \bar{\dif}(\px,t\bar{\pu}))\cdot A(\dif_0(\px,t\bar{\pu})) \cdot \bar{\dif}^{t}(\px,t\bar{\pu})),
\]
where $\widehat{\dif}=(\dif_1,\ldots,\dif_n)$.
By the Hadamard lemma $\dif_i(\px,\pu)=\sum\limits_{j=1}^{n} \dif_{ij}(\px,\pu)\pui{j}$ for some $\Cinfty$ functions $\dif_{ij}\colon\Wman\to\bR$ such that $\dif_{ij}(\px,0) = \ddd{}{\pui{j}}\dif_i$.
Let $J(\px,\pu)=(\dif_{ij}(\px,\pu))$ the matrix whose $i$th row consists of the functions $\dif_{i1},\ldots,\dif_{in}$.
Notice that $J(\px,0)$ is the Jacobi matrix of the composition
\[
    (\px\times\bR^n)\cap\Wman \xrightarrow{~\dif~} \Vman\times \bR^{n} \xrightarrow{~p_2~} \bR^{n}.
\]
Since $\dif$ is a diffeomorphism leaving invariant zero section $\Bman$, it follows that $J(\px,0)$ is non-degenerate.

One can also write $\widehat{\dif}(\px,\pu) = J(\px,\pu) \pu^{t}$, whence
\[
    \gamma(t) = t \bar{\pu} \cdot
                J(\px,t\bar{\pu})^{t} \cdot
                A(\dif_0(\px,t\bar{\pu})) \cdot
                J(\px,t\bar{\pu}) \cdot t \bar{\pu}^{t},
\]
which implies that
\[
    \delta(t) = \gamma(t)/t^2
                = \pu J(\px,t\pu)^{t} \cdot A(\dif_0(\px,t\bar{\pu})) \cdot J(\px,t\pu) \pu^{t}.
\]
Therefore,
\[
    \delta(0)
        = \pu \underbrace{J(\px,0)^{t} \cdot A(\dif_0(\px,0)) \cdot J(\px,0)}_{B(\px)} \pu^{t}
        = \pu B(\px) \pu^{t}.
\]
The assumption that $\func_1$ is definite means that $A(\dif_0(\px,0))$ is non-degenerate, whence $B$ is symmetric and non-degenerate as well, and therefore $\delta(0) =\pu B(\px) \pu^{t}>0$.

\ref{enum:sigma:homomorphism}
Suppose $\vbp_0=\vbp_1$ is the same vector bundle and $\func_0=\func_1\colon\Eman_0=\Eman_1 \to \bR$.
Since $\id_{\UInt}\circ\func_0 = \func_0\circ\id_{\ATor}$, it follows from uniqueness of $\smprM$, that $\smprM(\id_{\ATor})=\id_{\UInt}$.
Moreover, if $\dif,\dif'\in\FolDiff(\AFoliation_0)$, then
\begin{equation}\label{equ:sigma_unique}
    \smprM(\dif')\circ\smprM(\dif)\circ\func_0 =
    \smprM(\dif')\circ\func_0\circ\dif  =
    \func_0\circ\dif'\circ\dif =
    \smprM(\dif'\circ\dif)\circ\func_0.
\end{equation}
Now uniqueness of $\smprM$ implies that $\smprM(\dif')\circ\smprM(\dif)=\smprM(\dif'\circ\dif)$.
Thus, $\smprM$ is a homomorphism of monoids.
\end{proof}

\begin{example}\rm
Notice that the assumption that $\func_0$ and $\func_1$ are of the same homogeneity order is essential for $\phi$ to be a diffeomorphism.
Indeed, let $\func_0,\func_1\colon\Bman\times\bR^n\to[0;1]$ be given by $\func_0(\al,\px)=\dnrm{\px}^2$ and $\func_1(\px)=\dnrm{\px}^4$.
Then $\func_0$ is $2$-homogeneous, while $\func_1$ is $4$-homogeneous, $\AFoliation_0=\AFoliation_1$, and $\ATor_0=\func_0^{-1}(\UInt)=\func_1^{-1}(\UInt)=\ATor_1$.
Let also $\dif=\id_{\Bman\times\bR^n}$.
Then $\dif$ is a $(\AFoliation_0,\AFoliation_1)$-foliated diffeomorphism, while $\phi\colon\UInt\to\UInt$, $\phi(t)=t^2$, is a unique $\Cinfty$ map satisfying $\phi\circ\func_0=\func_1\circ\dif$.
However, $\phi$ is not a diffeomorphism.

It seems plausible that Lemma~\ref{lm:F0F1_folated} holds for definite homogeneous functions of the same degree $>2$, but one needs an analogue of Whitney Lemma~\ref{lm:Whitney_even_func} for ``evenness of higher order''.
\end{example}

\subsection*{Fiberwise definite $2$-homogeneous functions}
Let $\vbp\colon\Eman\to\Bman$ be a smooth vector bundle of rank $n$ over a compact manifold $\Bman$, $\func\colon\Eman\to\bR$ be a definite $2$-homogeneous $\Cinfty$ function, $\ATor = \func^{-1}(\UInt)$, and $\AFoliation = \{\func^{-1}(t) \mid t\in\UInt\}$ the partition of $\ATor$ into level sets of $\func$.
Then we have two homomorphisms $\smprA\colon\DiffPlI\to\DiffAFol$ and $\smprM\colon\DiffAFol \to \DiffPlI$, defined in Lemmas~\ref{lm:sigma:k-homog} and~\ref{lm:F0F1_folated} respectively, such that $\smprM(\dif)\circ\func=\func\circ\dif$ and $\rphi\circ\func=\func\circ\smprA(\rphi)$ for all $\dif\in\DiffAFol$ and $\rphi\in\DiffPlI$, so we have well-defined homomorphisms
\begin{equation}\label{equ:hat_sigma_theta}
\begin{aligned}
&\widehat{\smprM}\colon \DiffAFol \to \SMAf, && \widehat{\smprM}(\dif) = (\smprM(\dif),\dif), \\
&\widehat{\smprA}\colon \DiffPlI  \to \SMAf, && \widehat{\smprA}(\rphi) = (\rphi, \smprA(\rphi)).
\end{aligned}
\end{equation}

Notice that the foliation described in Theorem~\ref{th:Dlp_Dfol_he} corresponds to the definite homogeneous function $\func\colon\Circle\times\bR^2\to\bR$, $\func(\al,\px,\py)=\px^2+\py^2$, on the trivial vector bundle $\vbp\colon\Circle\times\bR^2\to\Circle$ of rank $2$ over the circle.
Therefore, Theorem~\ref{th:Dlp_Dfol_he} is a particular case of the following:
\begin{theorem}\label{th:def_2_hom_func}
The homomorphism
\begin{align*}
    &\prM\colon\SMRf \to \DiffAFol, &
    &\prM(\rphi,\dif)=\dif,
\end{align*}
induces an isomorphism of the following short exact sequences:
\begin{equation}\label{equ:diag_rel_Stabs:ext}
\begin{gathered}
\xymatrix{
1 \ar[r] &
\SMf \ar[r]^-{\sincl} \ar@{=}[d] &
\SMAf  \ar[rr]^-{\prA}           \ar@/_1ex/[d]_-{\prM} &&
\DiffPlI \ar@/^1ex/[ll]^-{\widehat{\smprA}}  \ar@{=}[d]            \ar[r] &  1 \\
1 \ar[r] &
\DiffAFolLp  \ar@{^(->}[r] &
\DiffAFol \ar@/_1ex/[u]_-{\widehat{\smprM}} \ar[rr]^-{\smprM} &&
\DiffPlI \ar@/^1ex/[ll]^-{\smprA} \ar[r] & 1
}
\end{gathered}
\end{equation}
and its inverse is $\widehat{\smprM}$.
Moreover, the pair $\bigl(\DiffAFolLp,\FolLpDiff(\AFoliation,\partial\ATor)\bigr)$ is a strong deformation retract of $\bigl(\DiffAFol,\FolDiff(\AFoliation,\partial\ATor)\bigr)$.
\end{theorem}
\begin{proof}
Evidently, $\prM\circ\widehat{\smprM}(\dif)=\dif$ for all $\dif\in\DiffAFol$, in particular, $\prM$ is surjective.
Since it is also injective, it follows that $\prM$ and $\widehat{\smprM}$ are mutually inverse isomorphisms of topological groups.
Moreover, by Lemma~\ref{lm:DAM_DM}, $\SMf=\DiffAFolLp$, and $\prM\circ\sincl = \id_{\SMf}$.
It is also evident that
\[
\prM(\SMAfd) = \Diff(\AFoliation,\partial\ATor).
\]

Now, due to Lemma~\ref{lm:sigma:k-homog}, the pair $\bigl(\SMf,\SMfd\bigr)$ is a strong deformation retract of $\bigl(\SMAf,\SMAfd\bigr)$ with respect to the inclusion $\sincl\colon\SMf\subset\SMAf$.
Hence, the pair $\bigl(\DiffAFolLp,\FolLpDiff(\AFoliation,\partial\ATor)\bigr)$ is a strong deformation retract of $\bigl(\DiffAFol,\FolDiff(\AFoliation,\partial\ATor)\bigr)$.
\end{proof}

\section{High-dimensional analogues of lens spaces}
\label{sect:proof:th:DiffLpq_fol_homtype}
In this section we prove Theorem~\ref{th:DFLlp_DPlFLfol} which includes Theorem~\ref{th:DiffLpq_fol_homtype_a} as a particular case.

For $i=0,1$ let $\vbp_i\colon\Eman_i\to\Bman_i$ be a smooth vector bundle of some rank $n_i$ over a compact manifold $\Bman_i$ and $\func_i\colon\Eman_i\to\bR$ be a definite $k_i$-homogeneous $\Cinfty$ function for some $k_i>0$.
Put
\begin{align*}
    \ATor_i  &:= \func_i^{-1}(\UInt),  &
    \AOTor_i &:= \func_i^{-1}\bigl([0;1)\bigr)=\ATor_i\setminus\partial\ATor_i, &
    \BLman   &:= \AOTor_0 \sqcup \AOTor_1.
\end{align*}
Denote by $\AFoliation_i = \{\func_i^{-1}(t) \mid t\in\UInt\}$ the partition of $\ATor_i$ into the level sets of $\func$.

\begin{lemma}\label{lm:properties_of_xi}
Suppose there exists a diffeomorphism $\psi\colon\partial\ATor_0 \to \partial\ATor_1$.
Then the map
\begin{align}\label{equ:gxi_gluing_map_general}
    &\gxi\colon\AOTor_0\setminus\Bman_0 \to \AOTor_1\setminus\Bman_1, &
    &\gxi(\px) = \rt{1-\func_0(\px)}{k_1} \cdot \psi\bigl(x / \rt{\func_0(\px)}{k_0}\bigr),
\end{align}
is a diffeomorphism such that
\begin{equation}\label{equ:xi_leaf_to_leaf}
    \func_0(\px) + \func_1(\gxi(\px)) = 1, \quad \px\in\AOTor_0\setminus\Bman_0,
\end{equation}
which is the same as $\gxi (\func_0^{-1}(t)) = \func_1^{-1}(1-t)$ for all $t\in(0;1)$, so $\gxi$ is a $(\AFoliation_0,\AFoliation_1)$-foliated diffeomorphism.
\end{lemma}
\begin{proof}
Let $\px\in\AOTor_0\setminus\Bman_0$, and $\py=\px/\rt{\func_0(\px)}{k_0}$.
Then
\[
    \func_0(\py) = \func_0\bigl(\px/\rt{\func_0(\px)}{k_0}\bigr)
                 = \func_0(\px)/\rt{\func_0(\px)}{k_0}^{k_0} = 1,
\]
i.e.\ $\py\in\partial\ATor_0$, so $\gxi$ is a well-defined map.
Moreover, since $\psi(\py)\in\partial\ATor_1$, $\func_1\bigl(\psi(\py)\bigr)=1$, and therefore
\begin{align*}
\func_1(\gxi(\px))
    &= \func_1\bigl(\rt{1-\func_0(\px)}{k_1} \cdot \psi(\py)\bigr)
     = \bigl(1-\func_0(\px)\bigr) \cdot \func_1\bigl(\psi(\py)\bigr) = 1-\func_0(\px).
% \qedhere
\end{align*}
Lemma is proved.
\end{proof}

\begin{example}\rm
Notice that even if $\partial\ATor_0$ and $\partial\ATor_1$ are diffeomorphic, the bases $\Bman_0$ and $\Bman_1$ may have distinct topological type.
The following example is inspired by surgery theory of manifolds.
Fix any $a,b\geq0$ and let $\vbp_0\colon S^{a}\times\bR^{b+1}\to S^{a}$ and $\vbp_1\colon S^{b}\times\bR^{a+1}\to S^{b}$ be trivial vector bundles over spheres $S^{a}$ and $S^{b}$.
Consider the following $2$-homogeneous functions $\func_0\colon S^{a}\times\bR^{b+1}\to\bR$ and $\func_1\colon S^{b}\times\bR^{a+1}\to\bR$ given by
\begin{align*}
    &\func_0(\px,\pui{1},\ldots,\pui{b+1}) = \sum_{i=1}^{b+1} \pui{i}^2, &
    &\func_1(\pv,\pvi{1},\ldots,\pvi{a+1}) = \sum_{i=1}^{a+1} \pvi{i}^2.
\end{align*}
Then both $\func_0^{-1}(1)$ and $\func_1^{-1}(1)$ are diffeomorphic with $S^{a}\times S^{b}$, though the bases $S^{a}$ and $S^{b}$ are not homeomorphic for $a\not=b$.
\end{example}

\begin{remark}\rm
Recall that each lens space $\Lman_{\gxi}$ is glued from two solid tori by some diffeomorphism $\gxi$ between their boundaries.
Though $\Lman_{\gxi}$ admits a smooth structure, such a definition has the following disadvantage: suppose we have a $\dif\colon\Lman_{\gxi}\to\Mman$ into some other manifold.
Therefore, even if we know that the restriction of $\dif$ on each tori $\ATor_i$ is $\Cinfty$, the entire map $\dif$ is not necessary even differentiable, and checking of its smoothness might be rather complicated.
For that reason, in what follows we will glue our analogues of lens spaces by diffeomorphism between open sets, as in Lemma~\ref{lm:properties_of_xi}.
\end{remark}

Suppose that there exists a diffeomorphism $\psi\colon\partial\ATor_0 \to \partial\ATor_1$.
Let $\PBLman = \AOTor_0 \cup_{\gxi} \AOTor_1$ be the space obtained by gluing $\AOTor_0$ with $\AOTor_1$ via the diffeomorphism $\gxi\colon\AOTor_0\setminus\Bman_0 \to \AOTor_1\setminus\Bman_1$ from Lemma~\ref{lm:properties_of_xi}.
Let also $\vbp\colon\BLman \to \PBLman$ be the corresponding quotient map and $\vbp_i:=\restr{\vbp}{\AOTor_i}\colon \AOTor_i \to \PBLman$ be the restriction maps.
Then it is evident that $\PBLman$ is a manifold, and each $\vbp_i$ is an open embedding.
One can regard the pair $\{\vbp_0,\vbp_1\}$ as a $\Cinfty$ atlas%
\footnote{\,Usually an atlas of a manifold $\Mman$ is a collection of open embeddings $\bR^{n} \supset \Uman_i \xrightarrow{\psi_i}\Mman$, $i\in\Lambda$, from open subsets of $\bR^{n}$ such that $\Mman = \cup_{i\in\Lambda} \psi_i(\Uman_i)$.
However, all the theory of manifolds will not be changed if one extends the notion of an atlas allowing each $\Uman_i$ to be an open subset of some $n$-manifold such that the corresponding transition functions are smooth maps.}
for $\PBLman$ with $\gxi$ being a transition map.
Define the following function
\[
   \hat{\func}\colon\BLman\to\UInt,
   \qquad
   \hat{\func}(\px) =
    \begin{cases}
        \func_0(\px),   & \px\in\AOTor_0,\\
        1-\func_1(\px), & \px\in\AOTor_1.
    \end{cases}
\]
It follows from~\eqref{equ:xi_leaf_to_leaf} that $\hat{\func}(\gxi(\px)) = \hat{\func}(\px)$ for all $\px\in\AOTor_{0}\setminus\Bman_0$, whence $\hat{\func}$ yields a well-defined $\Cinfty$ function $\func\colon\PBLman\to\UInt$ such that $\hat{\func}=\func\circ\vbp$.
It will be convenient to define the following diffeomorphism $\ophi\colon\UInt\to\UInt$, $\ophi(t)=1-t$.
Then we get the following commutative diagram:
\[
\xymatrix@R=1.2em@C=3em{
    \AOTor_0\setminus\Bman_0  \ar[dd]_-{\gxi}^{\cong} \ar@{^(->}[r] &
    \AOTor_0 \ar[r]^-{\func_0} \ar[d]_-{\vbp_0} &
    \UInt \ar@{=}[d]  \\
    & \PBLman \ar[r]^-{\func} & \UInt \\
    \AOTor_1\setminus\Bman_1  \ar@{^(->}[r] &
    \AOTor_1 \ar[u]^-{\vbp_1} \ar[r]^-{\func_1} &
    \UInt \ar[u]_-{\ophi}
}
\]

\begin{example}\label{ex:lens_spaces_case}\rm
If $\Eman_i = \Circle\times\bR^2\to\Circle$, $i=0,1$, are trivial vector bundles of rank $2$ over the circle, and the functions $\func_0=\func_1\colon\Circle\times\bR^2\to\bR$ coincide and are given by the formula $\func_0(\al,\px,\py) = \px^2+\py^2$, then the space $\PBLman$ is the same as lens space $\Lman_{\gxi}$, and one can assume that $\ATor_0 = \func^{-1}([0;\tfrac{1}{2}])$ and $\ATor_1 = \func^{-1}([\tfrac{1}{2};1])$.
In particular, Theorem~\ref{th:DiffLpq_fol_homtype} is a particular case of Theorem~\ref{th:DFLlp_DPlFLfol} below.
\end{example}

\begin{lemma}\label{lm:lens_gen:theta}
There exists a continuous homomorphism
\[
    \smprA\colon\DiffPlI\to\DiffPlAFol
\]
such that $\phi\circ\func = \func\circ\smprA(\dif)$ for all $\phi\in\DiffPlI$.
Therefore $\SMf$ is a strong deformation retract of $\SPlMAf$ with respect to the natural inclusion $\sincl\colon\SMf\subset\SMAf$.
\end{lemma}
\begin{proof}
By definition, $\func_i\colon\Eman\to\bR$, $i=0,1$, is a $k_i$-homogeneous $\Cinfty$ function, so it satisfies the property \JProp\ at each $\px\in\Bman_i$.
Moreover, $\func_0$ and $\func_1$ are local representations of $\func$ in the ``local chart'' $\AOTor_0$ and $\AOTor_1$, it follows that \term{$\func$ has property \JProp\ at each $\px\in\leaf{0}\cup\leaf{1}$}.
Now the result follows from Theorem~\ref{th:SMf_SMRf_he}.
However, we will give an explicit proof similar to the proof of Lemma~\ref{lm:sigma:k-homog}.

1) Let $\phi\in\DiffPlI$.
It will be convenient to put $\phi_0=\phi$.
As $\phi_0\bigl([0;1)\bigr)=[0;1)$, it follows from Lemma~\ref{lm:sigma:k-homog} that the map
\begin{align*}
    &\dif_0\colon\AOTor_0\to\AOTor_0, &
    &\dif_0(\px) = \rt{\gfunc_0(\func_0(\px))}{k_0}\,\px,
\end{align*}
is a diffeomorphism satisfying $\phi_0\circ\func_0 = \func_0\circ\dif_0$, where $\gfunc_0\colon\UInt\to(0;+\infty)$ is a unique $\Cinfty$ function such that $\phi_0(t) = \gfunc_0(t)t$.

Similarly, we have that $\phi\bigl((0;1]\bigr)=(0;1]$.
Consider another diffeomorphism
$\phi_1\in\DiffPlI$, $\phi_1(t) = \ophi\circ\phi\circ\ophi(t)=1-\phi(1-t)$.
Then $\phi_1\bigl([0;1)\bigr)=[0;1)$, whence, again by Lemma~\ref{lm:sigma:k-homog}, the map
\begin{align*}
    &\dif_1\colon\AOTor_1\to\AOTor_1, &
    &\dif_1(\px) = \rt{\gfunc_1(\func_1(\px))}{k_1}\,\px,
\end{align*}
is a diffeomorphism satisfying $\phi_1\circ\func_1 = \func_1\circ\dif_1$, where $\gfunc_1\colon\UInt\to(0;+\infty)$ is a unique $\Cinfty$ function such that $\phi_1(t) = \gfunc_1(t)t$.

We claim that
\begin{equation}\label{equ:gxi_h0__h1_gxi}
    \gxi\circ\dif_0(\px) = \dif_1\circ\gxi(\px), \quad \px\in\AOTor_0\setminus\Bman_0.
\end{equation}
This will imply that $\dif_0$ and $\dif_1$ yield a well-defined diffeomorphism
\begin{align*}
&\smprA(\phi)\colon\PBLman\to\PBLman,
&
&\smprA(\phi)(\px) =
\begin{cases}
\dif_0(\vbp_0^{-1}(\px)), & \px\in \vbp_0(\AOTor_0), \\
\dif_1(\vbp_1^{-1}(\px)), & \px\in \vbp_1(\AOTor_1).
\end{cases}
\end{align*}
Moreover, we will also have that $\phi\circ\func=\func\circ\smprA(\phi)$.

Before proving~\eqref{equ:gxi_h0__h1_gxi} let us establish several simple identities:
\begin{gather}
\label{equ:h0_dif_root_fh0}
\begin{aligned}
\frac{\dif_0(\px)}
{\rt{\func_0\circ\dif_0(\px)}{k_0}}
&=
\frac{\rt{\gfunc_0(\func_0(\px))}{k_0}\,\px}
{\rt{\func_0\bigl(\rt{\gfunc_0(\func_0(\px))}{k_0}\,\px\bigr)}{k_0}}
 \\ &=
\sqrt[k_0]{
    \frac{\gfunc_0(\func_0(\px))}
    {\rt{\gfunc_0(\func_0(\px))}{k_0}^{k_0}\func_0(\px)}
} \, \px
=
\frac{\px}{\rt{\func_0(\px)}{k_0}},
\end{aligned}\\[1mm]
%%%%%%%%%%
\label{equ:g1_q_f0__q_f0}
\begin{aligned}
(\gfunc_1\circ\ophi\circ\func_0(\px)) \cdot (\ophi\circ\func_0(\px)) &=
\phi_1\circ\ophi\circ\func_0(\px) \\ &=
\ophi\circ\phi_0\circ\func_0(\px) =
\ophi\circ\func_0\circ\dif_0(\px).
\end{aligned}
\end{gather}
Then
\begin{align*}
\dif_1\circ\gxi(\px) &=
\rt{\gfunc_1\circ\func_1\circ\gxi(\px)}{k_1}\cdot \gxi(\px)
\stackrel{\eqref{equ:xi_leaf_to_leaf},\eqref{equ:gxi_gluing_map_general}}{=\!=}\\
&= \rt{\gfunc_1\circ\ophi\circ\func_0(\px)}{k_1} \cdot
\rt{\ophi\circ\func_0(\px)}{k_1} \cdot \psi\bigl(x / \rt{\func_0(\px)}{k_0}\bigr)
\stackrel{\eqref{equ:g1_q_f0__q_f0}}{=} \\
&=\rt{\ophi\circ\func_0\circ\dif_0(\px)}{k_1} \cdot \psi\bigl(x / \rt{\func_0(\px)}{k_0}\bigr)
\stackrel{\eqref{equ:h0_dif_root_fh0}}{=} \\
&=\rt{\ophi\circ\func_0\circ\dif_0(\px)}{k_1}\cdot
\psi\bigl( \tfrac{\dif_0(\px)}{\rt{\func_0\circ\dif_0(\px)}{k_0}} \bigr)
\stackrel{\eqref{equ:gxi_gluing_map_general}}{=}
\gxi\circ\dif_0(\px).
\end{align*}
The correspondence $\phi\mapsto\smprA(\phi)$ is a homomorphism, since so is the correspondence $\phi_0\mapsto\dif_0$.
Continuity of $\smprA$ follows from the formulas for $\dif_0$ and $\dif_1$.

2) The proof that $\SMf$ is a strong deformation retract of $\SPlMAf$ is literally the same as in step 4) of the proof of Lemma~\ref{lm:sigma:k-homog}.
\end{proof}

For $t\in\UInt$ let $\leaf{t}:= \func^{-1}(t)$, and $\PAFoliation = \{ \leaf{t} \}_{t\in\UInt}$ be the foliation on $\BLman$ by level sets of $\func$.
Then $\leaf{t} = \vbp(\func_0^{-1}(t))$ for $t\in[0;1)$ and $\leaf{t} = \vbp(\func_1^{-1}(1-t))$ for $t\in(0;1]$.
In particular, $\leaf{0} = \vbp(\Bman_0)$ and $\leaf{1} = \vbp(\Bman_1)$ are ``singular leaves'', and each $\leaf{t}$, $t\in(0;1)$, is diffeomorphic with $\partial\ATor_0$.
Let $\DiffPlAFol$ be the group of $\PAFoliation$-foliated diffeomorphisms leaving invariant each $\leaf{0}$ and $\leaf{1}$.

\begin{lemma}\label{lm:lens_gen:sigma}
Suppose $\func_0$ and $\func_1$ are definite and $2$-homogeneous.
Then there is a unique continuous homomorphism $\smprM\colon\DiffAFol \to \DiffI$ such that $\smprM(\dif)\circ\func = \func\circ\dif$, i.e.\ $(\smprM(\dif),\dif) \in \SMAf$.

Also, $\dif\in\DiffPlAFol$ if and only if $\smprM(\dif)\in\DiffPlI$, and moreover, $\smprM(\DiffPlAFol)=\DiffPlI$.
Hence, if $\DiffPlAFol\not=\DiffAFol$, then $\smprM$ is surjective.
\end{lemma}
\begin{proof}
1) Let $\dif\in\DiffAFol$.
We should construct a diffeomorphism $\phi\colon\UInt\to\UInt$ satisfying $\phi\circ\func=\func\circ\dif$ and then put $\smprM(\dif)$.
As $\dif$ is $\AFoliation$-foliated, for each $t\in\UInt$ there exists a unique $t'\in\UInt$ such that $\dif(\leaf{t})=\leaf{t'}$.
Then, by condition~\ref{enum:LR3}, $\phi$ must be defined by $\phi(t)=t'$.
In particular, $\phi$ is uniquely determined by $\dif$, and we need to check that $\phi$ is a diffeomorphism.
Consider two cases.

a) First assume that $\dif\in\DiffPlAFol$, i.e.\ $\dif(\leaf{i})=\leaf{i}$ for $i=0,1$.
Then $\vbp(\AOTor_i) = \PBLman\setminus\leaf{1-i}$ is also invariant under $\dif$.
Hence $\dif$ yields a $\AFoliation_{i}$-foliated diffeomorphism
\[
    \dif_i= \vbp_i\circ\dif\circ\vbp_i\colon
    \AOTor_i\xrightarrow{\vbp_i} \vbp(\AOTor_i) \xrightarrow{\dif} \vbp(\AOTor_i) \xrightarrow{\vbp_i^{-1}}\AOTor_i, \quad i=0,1,
\]
such that $\dif_1\circ\gxi = \gxi\circ\dif_0$.
Then by Lemma~\ref{lm:F0F1_folated}, there exists a unique diffeomorphism $\phi_i\colon[0;1)\to[0;1)$ such that $\phi_i\circ\func_i=\func_i\circ\dif_i$.
Thus we get the following commutative diagram:
\[
\xymatrix@R=1.2em@C=4em{
    & &   \ar@{-->}@(l,l)[ddd]^{\phi} \UInt  \ar@{-->}@(r,r)[ddd]_{\phi} \\
    \UInt \ar[d]_-{\phi_0} \ar@/^1em/@{=}[urr] &
    \AOTor_0 \ar[d]^-{\dif_0} \ar[r]^-{\vbp_0} \ar[l]_-{\func_0} &
    \PBLman \ar[d]^-{\dif} \ar[u]^-{\func} &
    \AOTor_1 \ar[l]_-{\vbp_1} \ar[d]^-{\dif_1} \ar[r]^-{\func_1} &
    \UInt \ar[d]^-{\phi_1} \ar@/_1em/[ull]^-{\ophi}
    \\
    \UInt  \ar@/_1em/@{=}[drr] &
    \AOTor_0 \ar[r]^-{\vbp_0} \ar[l]_-{\func_0} &
    \PBLman  \ar[d]^-{\func}  &
    \AOTor_1 \ar[l]_-{\vbp_1}  \ar[r]^-{\func_1} &
    \UInt \ar@/^1em/[dll]_-{\ophi} \\
    & & \UInt
}
\]
It implies that $\phi_0=\ophi\circ\phi_1\circ\ophi^{-1}$, i.e.\ $\phi_0(t)=1-\phi_1(1-t)$ for $t\in(0;1)$.
Therefore, we get a well defined diffeomorphism $\phi\in\DiffPlI$ given by either of the formulas:
\begin{equation}\label{equ:phi0__1_phi1}
    \phi(t) =
    \begin{cases}
        \phi_0(t),        & t\in[0;1), \\
        1 - \phi_1(1-t),  & t\in(0;1]
    \end{cases}
\end{equation}
and satisfying $\phi\circ\func=\func\circ\dif$.

b) Suppose that $\dif(\leaf{i})=\leaf{1-i}$ for $i=0,1$.
Then $\dif$ yields a $(\AFoliation_{i},\AFoliation_{1-i})$-foliated diffeomorphism
\[
    \dif_i= \vbp_{1-i}\circ\dif\circ\vbp_i\colon
    \AOTor_i\xrightarrow{\vbp_i} \vbp(\AOTor_i) \xrightarrow{\dif} \vbp(\AOTor_{1-i}) \xrightarrow{\vbp_{1-i}}\AOTor_{1-i},
    \quad
    i=0,1,
\]
such that $\gxi^{-1}\circ\dif_0 = \dif_1\circ\gxi$.
Therefore, by Lemma~\ref{lm:F0F1_folated}, there exists a diffeomorphism $\phi_i\colon[0;1)\to[0;1)$ such that $\phi_i\circ\func_i=\func_{1-i}\circ\dif_i$.
Thus we get the following commutative diagram:
\[
\xymatrix@R=1.2em@C=4em{
    & & \ar@{-->}@(l,l)[ddd]^{\phi} \UInt  \ar@{-->}@(r,r)[ddd]_{\phi} \\
    \UInt \ar[d]_-{\phi_0} \ar@/^1em/@{=}[urr] &
    \AOTor_0 \ar[d]^-{\dif_0} \ar[r]^-{\vbp_0} \ar[l]_-{\func_0} &
    \PBLman  \ar[d]^-{\dif} \ar[u]^{\func} &
    \AOTor_1 \ar[l]_-{\vbp_1}  \ar[d]^-{\dif_1} \ar[r]^-{\func_1} &
    \UInt \ar[d]^-{\phi_1} \ar@/_1em/[ull]^-{\ophi}
    \\
    \UInt  \ar@/_1em/[drr]_-{\ophi} &
    \AOTor_1 \ar[r]^-{\vbp_1} \ar[l]_-{\func_1} &
    \PBLman  \ar[d]^-{\func}  &
    \AOTor_0 \ar[l]_-{\vbp_0} \ar[r]^-{\func_0} &
    \UInt  \ar@{=}@/^1em/[dll] \\
    & & \UInt
}
\]
In particular, $\ophi\circ\phi_0 = \phi_1\circ\ophi$, i.e.\ $1 - \phi_0(t) = \phi_1(1-t)$ for $t\in(0;1)$.
Therefore, we get a well defined reversing orientation diffeomorphism $\phi\in\DiffI$ given by
\begin{equation}\label{equ:1_phi0__phi1}
    \phi(t) =
    \begin{cases}
        \ophi\circ\phi_0(t) = 1 - \phi_0(t),  & t\in[0;1), \\
        \phi_1\circ\ophi(t) = \phi_1(1-t),    & t\in(0;1].
    \end{cases}
\end{equation}
and satisfying $\phi\circ\func=\func\circ\dif$.

2) Due to Lemma~\ref{lm:F0F1_folated} and formulas for $\phi$, the correspondence $\dif\mapsto\phi$ is a well-defined continuous map $\smprM\colon\DiffAFol\to\DiffI$, $\smprM(\dif)=\phi$.
Moreover, as in~\eqref{equ:sigma_unique}, uniqueness of $\smprM$ implies that $\smprM$ is a homomorphism.
Also, by the construction $\smprM(\dif)\in\DiffPlI$ if and only if $\dif\in\DiffPlAFol$.

Further notice, that due to Lemma~\ref{lm:lens_gen:theta}, $\smprM(\smprA(\phi))=\phi$ for all $\phi\in\DiffPlI$.
Indeed, we have that $\phi\circ\func = \func\circ\smprA(\phi)$, and $\smprM(\smprA(\phi))\circ\func=\func\circ\smprA(\phi)$.
Then by uniqueness of $\smprM$ we should have that $\smprM(\smprA(\phi))=\phi$.
This implies that $\smprM(\DiffPlAFol)=\DiffPlI$.

Finally, suppose that $\DiffPlAFol\not=\DiffAFol$.
Then $\smprM$ must map the unique adjacent class $\DiffAFol\setminus\DiffPlAFol$ of $\DiffAFol$ by $\DiffPlAFol$, onto the unique adjacent class $\DiffI\setminus\DiffPlI$ of $\DiffI$ by $\DiffPlI$.
Hence $\smprM$ is surjective.
\end{proof}

\begin{theorem}\label{th:DFLlp_DPlFLfol}
Suppose $\func_0$ and $\func_1$ are definite and $2$-homogeneous.
Then the projection $\prM\colon\SMAf\to\DiffAFol$ induces an isomorphism of the following short exact sequences:
\begin{equation}\label{equ:gen_lens:diag_rel_Stabs:ext_plus}
\begin{gathered}
\xymatrix{
1 \ar[r] &
\SMf \ar[r]^-{\sincl} \ar@{=}[d] &
\SPlMAf  \ar[r]^-{\prA}           \ar@/_1ex/[d]_-{\prM} &
\DiffPlI \ar@/^1ex/[l]^-{\widehat{\smprA}}  \ar@{=}[d]            \ar[r] &  1 \\
1 \ar[r] &
\DiffAFolLp  \ar@{^(->}[r] &
\DiffPlAFol \ar@/_1ex/[u]_-{\widehat{\smprM}} \ar[r]^-{\smprM} &
\DiffPlI \ar@/^1ex/[l]^-{\smprA} \ar[r] & 1
}
\end{gathered}
\end{equation}
where $\widehat{\smprM}(\dif) = (\smprM(\dif),\dif)$ is the inverse to $\prM$, and $\widehat{\smprA}(\rphi) = (\rphi, \smprA(\rphi))$ is the inverse to $\prA$.
In particular, $\DiffAFolLp$ is a strong deformation retract of $\DiffPlAFol$.

Moreover, if $\DiffPlAFol \not= \DiffAFol$, then we have isomorphism of another pair of short exact sequences:
\begin{equation}\label{equ:gen_lens:diag_rel_Stabs:ext}
\begin{gathered}
\xymatrix{
1 \ar[r] &
\SMf \ar[r]^-{\sincl} \ar@{=}[d] &
\SMAf  \ar[r]^-{\prA}           \ar@/_1ex/[d]_-{\prM} &
\DiffI   \ar@{=}[d]            \ar[r] &  1 \\
1 \ar[r] &
\DiffAFolLp  \ar@{^(->}[r] &
\DiffAFol \ar@/_1ex/[u]_-{\widehat{\smprM}} \ar[r]^-{\smprM} &
\DiffI \ar[r] & 1
}
\end{gathered}
\end{equation}
\end{theorem}
\begin{proof}
Evidently, $\prM\circ\widehat{\smprM}(\dif)=\dif$ for all $\dif\in\DiffAFol$, in particular, $\prM$ is surjective and maps $\SPlMAf$ onto $\DiffPlAFol$.
Since it is also injective (see Lemma~\ref{lm:DAM_DM}), it follows that $\prM$ and $\widehat{\smprM}$ are mutually inverse isomorphisms of topological groups.
Moreover, again by Lemma~\ref{lm:DAM_DM}, $\SMf=\DiffAFolLp$ and $\prM\circ\sincl = \id_{\SMf}$.
Hence $\prM$ yields isomorphisms of the corresponding short exact sequences in~\eqref{equ:gen_lens:diag_rel_Stabs:ext_plus} and~\eqref{equ:gen_lens:diag_rel_Stabs:ext}.

By Lemma~\ref{lm:lens_gen:theta}, $\SMf$ is a strong deformation retract of $\SMAf$ with respect to the inclusion $\sincl\colon\SMf\subset\SMAf$, whence $\DiffAFolLp$ is a strong deformation retract of $\DiffPlAFol$.
\end{proof}

\section{Solid torus}\label{sect:main_result:solid_torus}
Let $\AFoliation$ be the foliation on $\ATor = \Circle\times D^2$ into the central circle and parallel tori, as in Section~\ref{sect:main_result:all}.
Then, by Theorem~\ref{th:Dlp_Dfol_he}, $\bigl(\DiffAFolLp, \ \FolLpDiffdATor\bigr)$ is a strong deformation retract of the pair $\bigl(\DiffAFol, \ \FolDiffdATor\bigr)$.
In this section we will compute the homotopy types of all these groups, see Theorem~\ref{th:hom_type_solid_torus_groups}.

Consider the following subgroup of $\GL(2,\bZ)$ consisting of matrices for which the vector $\avect{0}{1}$ is eigen with eigen value $\pm 1$:
\begin{equation}\label{equ:group_A}
    \MPGSolidTorus:=\left\{ \amatr{\eps}{0}{m}{\delta} \mid m\in\bZ, \, \eps,\delta\in\{\pm1\} \right\}.
\end{equation}
Then $\MPGSolidTorus$ is generated by the matrices:
\begin{align}\label{equ:dlmt}
    \mDtwist  &= \amatr{1}{0}{1}{1},  &
    \mLambda  &= \amatr{-1}{0}{0}{1}, &
    \mMu      &= \amatr{1}{0}{0}{-1} = -\hLambda, &
    \mTau     &= \amatr{-1}{0}{0}{-1},
\end{align}
satisfying the following identities:
\begin{align}\label{equ:dlmt_relations}
    &\mLambda^2 = \mMu^2 = E, &
    &\mLambda\mDtwist\mLambda = \mMu\mDtwist\mMu = \mDtwist^{-1}, &
    &\mTau = \mLambda\mMu = \mMu\mLambda, &
    &\mTau\mDtwist=\mDtwist\mTau.
\end{align}
Hence, $\MPGSolidTorus = \langle \mDtwist, \mLambda \rangle \times  \langle \mTau\rangle \cong \Dih(\bZ) \times \bZ_2$.
It is easy to check that every $\mAmatr=\amatr{\eps}{0}{m}{\delta}\in\MPGSolidTorus$ yields the following diffeomorphism $\diftor{\mAmatr}\in\DiffAFolLp$,
\begin{equation}\label{equ:g_A}
    \diftor{\mAmatr}(\al,\az) =
    \begin{cases}
        \bigl( \al^{\eps}, \al^{m} \az \bigr),       & \text{if} \ \delta = 1, \\
        \bigl( \al^{\eps}, \al^{m} \bar{\az} \bigr), & \text{if} \ \delta = -1,
    \end{cases}
\end{equation}
so that the correspondence $\mAmatr \mapsto \diftor{\mAmatr}$ is a monomorphism $\MPGSolidTorus\subset\DiffAFolLp$, and we will identify $\MPGSolidTorus$ with its image in $\DiffAFolLp$.
It will be convenient to denote $\hDtwist := \diftor{\mDtwist}$, $\hLambda := \diftor{\mLambda}$, $\hMu := \diftor{\mMu}$, $\hTau := \diftor{\mTau}$, so
\begin{equation}\label{equ:dlmt_diffs}
\begin{aligned}
    \hDtwist(\al,\az)   &= (\al, \al\az),         &
    \hLambda(\al,\az)   &= (\bar{\al}, \az),      \\
    \hMu(\al,\az)       &= (\al,\bar{\az}),       &
    \hTau(\al,\az)      &= (\bar{\al},\bar{\az}).
\end{aligned}
\end{equation}
Further note that $\DiffAFolLp$ contains another subgroup $\RotSub$, called the \term{rotation subgroup}, isomorphic to the $2$-torus $\Circle\times\Circle$ and consisting of diffeomorphisms $\hRho{\alpha,\beta}$, $(\alpha,\beta)\in \Circle\times\Circle$, given by
\begin{equation}\label{equ:incl:rotsubgr}
    \hRho{\alpha,\beta}(\al,\az) = (\alpha\al, \beta\az), \qquad (\al,\az)\in \ATor = \Circle\times D^2.
\end{equation}

Let $\AffDiffATor = \langle \RotSub, \MPGSolidTorus\rangle$ be the subgroup of $\DiffAFolLp$ generated by $\RotSub$ and $\MPGSolidTorus$.
Evidently, $\RotSub$ is the identity path component of $\AffDiffATor$, whence $\pi_0\AffDiffATor \cong \MPGSolidTorus$.

\begin{theorem}\label{th:hom_type_solid_torus_groups}
In the following diagrams of inclusions the arrows denoted by (w.)h.e. are (weak) homotopy equivalences (the new statements are written in bold):
\begin{align*}
&\xymatrix@R=1.5em{
    \{\id_{\ATor}\} \ar@{^(->}[rr]^-{w.h.e} \ar@{^(->}[d]_-{h.e.} &&
    \FolLpDiffdATor  \ar@{_(->}[d]^-{\textbf{h.e}} \\
    \DiffdATor && \FolDiffdATor \ar@{_(->}[ll]_{\textbf{w.h.e}}
}
&&
\xymatrix@R=1.5em{
    \AffDiffATor \ar@{^(->}[rr]^-{w.h.e} \ar@{^(->}[d]_-{h.e.} &&
    \DiffAFolLp  \ar@{_(->}[d]^-{\textbf{h.e}} \\
    \DiffATor && \DiffAFol \ar@{_(->}[ll]_{\textbf{w.h.e}}
}
\end{align*}
\end{theorem}
\begin{proof}
The proof that the inclusion $\{\id_{\ATor}\} \subset \DiffdATor$ is a homotopy equivalence, i.e.\ contractibility of $\DiffdATor$, is given in~\cite[Theorem~2]{Ivanov:ST:1982}.
The fact, that the inclusion $\AffDiffATor \subset \DiffATor$ is a homotopy equivalence, is classical and follows from contractibility of $\DiffdATor$, see~\cite{KhokhliukMaksymenko:JHRS:2023} for exposition of those results.
Also for the proof that the corresponding induced map $\pi_0\AffDiffATor = \MPGSolidTorus \to \pi_0\DiffATor$ is an isomorphism see e.g.~\cite[Theorem~14]{Wajnryb:FM:1998}.
The upper horizontal weak homotopy equivalences are established in~\cite{KhokhliukMaksymenko:JHRS:2023}, and the right vertical homotopy equivalence are contained in Theorem~\ref{th:Dlp_Dfol_he}.
This implies that the lower arrows are weak homotopy equivalences as well.
\end{proof}

\section{Lens spaces}\label{sect:lens_spaces}
In this section we will compute the groups~\eqref{equ:diag:Lpq_all_groups:line} for all lens spaces, see Theorem~\ref{th:DiffLpq_fol_homtype}.
We assume that the reader is familiar with lens spaces and will briefly recall their basic properties, see e.g.~\cite{Reidemeister:AMSUH:1935, Brody:AM:1960, Bonahon:Top:1983, Gadgil:GT:2001, HongKalliongisMcCulloughRubinstein:LMN:2012}.

Let $\Lman:=\ATor\times\{0,1\}$ be two copies of $\ATor$ and $\xi\colon\partial\ATor\to\partial\ATor$ a diffeomorphism.
Identify each point $(\px,0)$ with $(\gxi(\px),1)$ and denote the obtained quotient space by $\Lman_{\gxi}$.
Let also $\vbp\colon\Lman \to \Lman_{\gxi}$ be the corresponding quotient map, $\ATor_i := \vbp(\ATor\times\{i\})$, $i=0,1$, and $\Cman_i := \Circle\times 0 \times i$ be the ``central circle'' of the torus $\ATor_i$.
It is well known that $\Lman_{\gxi}$ admits a smooth structure, and if $\gxi'\colon\partial\ATor\to\partial\ATor$ is another diffeomorphism satisfying either of the following conditions:
\begin{itemize}
\item $\gxi$ is isotopic to $\gxi'$;
\item $\gxi'=\restr{\gfunc}{\partial\ATor} \circ\gxi \circ\restr{\dif}{\partial\ATor}$, where $\gfunc,\dif$ are certain diffeomorphisms of $\ATor$;
\item $\gxi'=\gxi^{-1}$;
\end{itemize}
then $\Lman_{\gxi}$ and $\Lman_{\gxi'}$ are diffeomorphic.
This finally implies that one can always assume that $\gxi$ is defined by the formula:
\begin{equation}\label{equ:phi_pq}
\gxi(\al,\az)  = (\al^r \az^p, \al^s \az^q),  \quad (\al,\az)\in\partial\ATor,
\end{equation}
for some matrix $\mXimatr = \amatr{r}{p}{s}{q} \in\GL(2,\bZ)$ with $\nrm{\mXimatr} = rq-ps = -1$, and in this case $\Lman_{\gxi}$ will be denoted by $\Lpq{p}{q}$.
Moreover, if $\gxi'\colon\partial\ATor\to\partial\ATor$ is given by the matrix $\mXimatr'=\amatr{r'}{p}{s'}{q'}$ with $r=r'(\bmod\ p)$ and $q=q'(\bmod\ p)$, then $\Lman_{\gxi}$ and $\Lman_{\gxi'}$ are still diffeomorphic.
This implies that for $p=0,1,2$ there exists a unique (up to a diffeomorphism) lens space so that
\begin{itemize}
\item
$\Lpq{0}{1} \cong S^1\times S^2$ with $\mXimatr = \mLambda = \amatr{-1}{0}{0}{1}$ and $\gxi(\al,\az)=\hLambda(\al,\az)=(\bar{\al},\az)$;

\item
$\Lpq{1}{0} \cong S^3$ with $\mXimatr = \amatr{0}{1}{1}{0}$ and $\gxi(\al,\az)=(\az,\al)$;

\item
$\Lpq{2}{1} \cong \bR P^3$ with $\mXimatr = \amatr{1}{2}{1}{1}$ and $\gxi(\al,\az)=(\al\az^2,\al\az)$ or with $\mXimatr'=\mXimatr\mDtwist^{-1} = \amatr{-1}{2}{0}{1}$ and $\gxi'(\al,\az)=(\bar{\al}\az^2,\az)$.
\end{itemize}
For $p>2$ one can assume that $1\leq q < p$, $\mathrm{gcd}(p,q)=1$, and $q$ can be replaced with $q'$ such that $qq'=1(\bmod\ p)$.

\subsection*{Subgroups $\ADiffLpq{p,q}$ and $\AFolDiffLpq{p,q}$ of $\Diff(\Lpq{p}{q})$}
We will now recall the definition of a certain subgroup $\ADiffLpq{p,q} \subset \Diff(\Lpq{p}{q})$ considered in~\cite{KhokhliukMaksymenko:JHRS:2023}, and also define some other group $\AFolDiffLpq{p,q}$ containing $\ADiffLpq{p,q}$, see item~\ref{enum:Lpq_groups} below.
They will play the principal role in Theorem~\ref{th:DiffLpq_fol_homtype}.
\begin{enumerate}[wide, label={\rm\arabic*)}, itemsep=1ex]
\item
First note that every diffeomorphism $\phi\in\Diff(\Lman_{\gxi})$ which preserves the common boundary $\partial\ATor_0=\partial\ATor_1$ induces a diffeomorphism $\widehat{\phi}\colon\Lman\to\Lman$ such that $\vbp\circ\widehat{\phi} = \phi\circ\vbp\colon\Lman\to\Lman_{\gxi}$.

a) Moreover, if $\phi$ preserves each torus $\ATor_i$, $i=0,1$, then $\restr{\widehat{\phi}}{\partial\ATor\times\{i\}}(\px,i) = (\phi_i(\px),i)$ for unique diffeomorphisms $\phi_0,\phi_1\colon\ATor\to\ATor$ satisfying the identity $\restr{\gxi\circ\phi_0}{\partial\ATor} = \phi_1\circ\gxi\colon\partial\ATor\to\partial\ATor$, i.e. making commutative the diagram (a) in~\eqref{equ:commuting_with_xi}:
\begin{equation}\label{equ:commuting_with_xi}
\begin{gathered}
\begin{array}{ccc}
\xymatrix@R=1.5em{
\partial\ATor \ar[d]_-{\gxi} \ar[r]^-{\phi_0} &
\partial\ATor \ar[d]^-{\gxi} \\
\partial\ATor \ar[r]^-{\phi_1} &
\partial\ATor
}
&\qquad\qquad&
\xymatrix@R=1.5em{
\partial\ATor \ar[d]_-{\gxi} \ar[r]^-{\phi_0} &
\partial\ATor  \\
\partial\ATor \ar[r]^-{\phi_1} &
\partial\ATor  \ar[u]_-{\gxi}
} \\
(a) &&
(b)
\end{array}
\end{gathered}
\end{equation}
It will be convenient to write down $\phi$ as follows: $\diflpdir{\phi}{\phi_0}{\phi_1}$.
Evidently, $\diflpdir{\phi^{-1}}{\phi_0^{-1}}{\phi_1^{-1}}$, and if $\phi\in\FolLpDiffLpq{p}{q}$, then $\phi_0,\phi_1\in\DiffAFolLp$.

b) Similarly, if $\phi$ exchanges $\ATor_0$ and $\ATor_1$, then $\restr{\widehat{\phi}}{\partial\ATor\times\{i\}}(\px,i) = (\phi_i(\px),1-i)$ for unique diffeomorphisms $\phi_0,\phi_1\colon\ATor\to\ATor$ satisfying the identity $\restr{\phi_0}{\partial\ATor} = \gxi \circ \phi_1\circ\gxi\colon\partial\ATor\to\partial\ATor$, i.e. making commutative the diagram (b) in~\eqref{equ:commuting_with_xi}.
In that case we will write down $\phi$ as follows: $\diflprew{\phi}{\phi_0}{\phi_1}$.
Notice that then $\diflprew{\phi^{-1}}{\phi_1^{-1}}{\phi_0^{-1}}$.

\item
For $(\alpha,\beta)\in\Circle\times\Circle$ let $\hRho{\alpha,\beta}\colon\ATor\to\ATor$, $\hRho{\alpha,\beta}(\al,\az) = (\alpha\al, \beta\az)$, be the diffeomorphism given by~\eqref{equ:incl:rotsubgr}.
Evidently,
\[
    \gxi\circ \hRho{\alpha,\beta}(\al,\az) = (\alpha^{r}\beta^{p}\al^r\az^p,\alpha^{s}\beta^{q}\al^s\az^q)=\hRho{\gxi(\alpha,\beta)} \circ \gxi(\al,\az),
\]
and we have the following diffeomorphism $\diflpdir[3em]{\lRot{\alpha,\beta}}{ \hRho{\alpha,\beta} }{ \hRho{\gxi(\alpha,\beta)} }$ of $\Lpq{p}{q}$, called a \term{rotation}.
Then the group $\RotSub = \{ \hRho{\alpha,\beta} \mid \alpha,\beta\in\Circle\} \subset \Diff(\Lpq{p}{q})$ of all rotations is isomorphic to $\Circle\times\Circle=\partial\ATor$.

\item
We will recall now the definition of two diffeomorphisms $\lSigmaPl$ and $\lSigmaMin$ exchanging $\ATor_0$ and $\ATor_1$.
F.~Bonahon~\cite[Theorem~1]{Bonahon:Top:1983} proved that every diffeomorphism $\dif$ of $\Lpq{p}{q}$ is isotopic to a diffeomorphism which leaves the common boundary $\ATor_0$ and $\ATor_1$ invariant, and thus preserving or exchanging those tori.
Moreover, for $p>2$ the group $\pi_0\Diff(\Lpq{p}{q})$ is generated by the isotopy classes of the diffeomorphism $\diflpdir{\lTau}{\hTau}{\hTau}$, and $\lSigmaPl$, $\lSigmaMin$ (more precisely by that $\lSigmaMin$ or $\lSigmaPl$ which is defined for the given $(p,q)$).

\begin{enumerate}[wide, label={\rm\alph*)}, itemsep=1ex]
\item
Suppose $-q^2 - ps = -1$ for some $s\in\bZ$, so $\gxi$ is given by the matrix $\mXimatr=\amatr{-q}{p}{s}{q}$.
This property holds when $\Lpq{p}{q}$ is either $\Lpq{0}{1}=S^1\times S^2$, or $\Lpq{1}{0}=S^3$, or $\Lpq{2}{1}=\RP{3}$, or $q^2=1(\bmod\ p)$ for $p>2$.
Notice that $\mXimatr^2=E$, so $\gxi^2=\id_{\partial\ATor}$, i.e.\ $\id_{\partial\ATor} = \gxi\circ\id_{\partial\ATor}\circ\gxi$, and we have a well defined diffeomorphism $\diflprew{\lSigmaPl}{\id_{\ATor}}{\id_{\ATor}}$ of $\Lpq{p}{q}$ exchanging $\ATor_0$ and $\ATor_1$.
One easily checks that $\lSigmaPl$ preserves orientation and $\lSigmaPl^2=\id_{\Lpq{p}{q}}$.

\item
Suppose $q^2 - ps = -1$ for some $s\in\bZ$, so $\gxi$ is given by the matrix $\mXimatr=\amatr{q}{p}{s}{q}$.
This holds if $\Lpq{p}{q}$ is either $\Lpq{1}{0}=S^3$, or $\Lpq{2}{1}=\RP{3}$ or if $q^2=-1(\bmod\ p)$ for $p>2$.
One easily checks that $\mLambda = \mXimatr\mMu\mXimatr$, so $\hLambda=\gxi\circ\hMu\circ\gxi$, and we have another diffeomorphism $\diflprew{\lSigmaMin}{\hLambda}{\hMu}$ of $\Lpq{p}{q}$ exchanging $\ATor_0$ and $\ATor_1$.
One easily checks that $\lSigmaMin$ reverses orientation, and is periodic of period $4$.
\end{enumerate}

\item\label{enum:Lpq_groups}
We will now define the group $\ADiffLpq{p,q}$ and put $\AFolDiffLpq{p,q}:=\langle\ADiffLpq{p,q},\lSigmaMin,\lSigmaPl\rangle$ to be the group generated by
$\ADiffLpq{p,q}$ and those of $\lSigmaMin$ and $\lSigmaPl$ which are defined for the given values $(p,q)$.

Let $\hDtwist,\hLambda,\hMu,\hTau=\hLambda\circ\hMu\in\DiffAFolLp$ be diffeomorphisms of $\ATor$ given by~\eqref{equ:dlmt_diffs} and corresponding to matrices~\eqref{equ:dlmt}.
\begin{enumerate}[wide, label={\rm\alph*)}, itemsep=1ex]
\item
Let $\Lpq{p}{q} = \Lpq{0}{1} = S^1\times S^2$, so $\mXimatr=\mLambda=\amatr{-1}{0}{0}{1}$.
Then it follows from~\eqref{equ:dlmt_relations} that $\Diff(\Lpq{0}{1})$ contains the following diffeomorphisms:
\begin{align*}
    & \diflpdir{\lDtwist}{\hDtwist}{\hDtwist^{-1}},
    &&\diflpdir{\lLambda}{\hLambda}{\hLambda},
    &&\diflpdir{\lMu}{\hMu}{\hMu},
    &&\diflpdir{\lTau=\lLambda\circ\lMu}{\hTau}{\hTau}.
\end{align*}
Notice that in this case only $\diflprew{\lSigmaMin}{\hLambda}{\hMu}$ is defined.
Define the following groups
\begin{align*}
    &\ADiffLpq{0,1} := \langle \RotSub, \lDtwist, \lLambda, \lMu \rangle \cong \langle \RotSub, \MPGSolidTorus\rangle = \AffDiffATor,
    &&\AFolDiffLpq{0,1} := \langle \ADiffLpq{p,q}, \lSigmaMin \rangle.
\end{align*}
Then $\RotSub$ is the identity path component of each of them.
Moreover, one also easily checks that
\[
    \diflprew[2em]{\lSigmaMin\circ\lDtwist=\lDtwist^{-1}\circ\lSigmaMin}{\hLambda\hDtwist}{\hMu\hDtwist^{-1}}
\]
and $\lSigmaMin$ commutes with $\lLambda$ and $\lMu$.
Then together with the first and second identities in~\eqref{equ:dlmt_relations} we get that
\begin{align*}
    &\pi_0\ADiffLpq{0,1} =  \MPGSolidTorus,
    &&\pi_0\AFolDiffLpq{0,1} = \langle \MPGSolidTorus, \lSigmaMin \rangle
    \cong
    \langle \lDtwist \rangle \rtimes
    \langle \lLambda,  \lMu,  \lSigmaMin \rangle
    \cong
    \bZ\rtimes(\bZ_2)^3,
\end{align*}
where the semidirect product corresponds to the homomorphism $(\bZ_2)^3 \to \Aut(\bZ) = \bZ_2$, $(a,b,c) \cdot n = (-1)^{a+b+c}n$, for $a,b,c\in\bZ_2$ and $n\in\bZ$.

\item
Let $\Lpq{p}{q} = \Lpq{1}{0} = S^3$, so $\mXimatr = \amatr{0}{1}{1}{0}$.
Then $\mXimatr^2=E$, $\mXimatr\mLambda=\mMu\mXimatr$, $\mXimatr\mMu=\mLambda\mXimatr$, which implies that $\FolLpDiffLpq{1}{0}$ contains the following diffeomorphisms of order two $\diflpdir{\lLambda}{\hLambda}{\hMu}$, $\diflpdir{\lMu}{\hMu}{\hLambda}$.
In this case both $\diflprew{\lSigmaMin}{\hLambda}{\hMu}$ and $\diflprew{\lSigmaPl}{\id_{\ATor}}{\id_{\ATor}}$ are defined, and one easily checks that
\begin{align*}
    &  \lSigmaPl^2=\id_{\Lpq{1}{0}},
    && \lSigmaMin^4=\id_{\Lpq{1}{0}},
    && \lSigmaMin^2 = \lLambda\lMu,
    && \lSigmaPl\lSigmaMin\lSigmaPl=\lSigmaMin^{-1}, \\
     & \lSigmaPl\lSigmaMin = \lLambda,
    && \lSigmaPl\lSigmaMin = \lMu,
    && \lLambda\hRho{\alpha,\beta} = \hRho{\bar{\alpha},\beta}\lLambda,
    && \lMu\hRho{\alpha,\beta} = \hRho{\alpha,\bar{\beta}}\lMu,
\end{align*}
Define the following groups
\begin{align*}
    &\ADiffLpq{0,1}     := \langle \RotSub, \lLambda, \lMu \rangle,
    &&\AFolDiffLpq{0,1} := \langle \ADiffLpq{0,1}, \lSigmaMin, \lSigmaPl \rangle.
\end{align*}
Then $\RotSub$ is the identity path component of each of them, and it follows from the above identities that
\begin{align*}
    &\ADiffLpq{0,1} =
    \langle\lRot{\alpha,1},\lLambda \mid \alpha\in\Circle\rangle
    \times
    \langle\lRot{1,\beta},\lMu \mid \beta\in\Circle\rangle =
    \Ort(2)\times \Ort(2), \\
    &\pi_0\ADiffLpq{1,0} = \langle \lLambda, \lMu \rangle = \bZ_2\times\bZ_2, \\
    &\AFolDiffLpq{0,1} :=
    \langle \RotSub, \lLambda, \lMu, \lSigmaMin, \lSigmaPl \rangle =
    \langle \RotSub, \lSigmaMin, \lSigmaPl \rangle, \\
    &\pi_0\AFolDiffLpq{0,1} = \langle \lSigmaMin, \lSigmaPl \rangle \cong \Dih(\bZ_4) = \bD_{4}.
\end{align*}

\item
Let $\Lpq{p}{q} = \Lpq{2}{1} = \RP{3}$, so $\mXimatr = \amatr{1}{2}{1}{1}$.
Then $\mXimatr\mLambda\mDtwist=\mDtwist\mMu\mXimatr=-\mXimatr$, $\mXimatr\mTau=\mTau\mXimatr$, and $\FolLpDiffLpq{2}{1}$ contains the following diffeomorphisms of order two:
$\diflpdir{\lTheta}{\hLambda\hDtwist}{\hDtwist\hMu}$ and $\diflpdir{\lTau}{\hTau}{\hTau}$.
Since $(p,q)=(2,1)$, we have that $q^2=-1(\bmod\ 2)$ and thus the diffeomorphism $\diflprew{\lSigmaMin}{\hLambda}{\hMu}$ is defined.

On the other hand, $\Lpq{2}{1}$ is also diffeomorphic to the lens space $\Lpq{2}{-1}$ given by the matrix $\mXimatr'=\mXimatr\mDtwist^{-1} = \amatr{-1}{2}{0}{1}$ for which $\diflprew{\lSigmaPl'}{\id}{\id}$ is also defined.
To transfer $\lSigmaPl'$ to a diffeomorphism of $\Lpq{2}{1}$ we need to ``conjugate'' it by $\hDtwist$.

More precisely, note that $E = \mXimatr'\mXimatr'= \mXimatr\mDtwist^{-1}\mXimatr\mDtwist^{-1}$, whence $\mDtwist = \mXimatr\mDtwist^{-1}\mXimatr$, and we get the following diffeomorphism $\diflprew{\lSigmaPl}{\hDtwist}{\hDtwist^{-1}}$ of $\Lpq{2}{1}$.
Define the following groups
\begin{align*}
    &\ADiffLpq{2,1} := \langle \RotSub, \lTheta, \lTau\rangle,
    &&\AFolDiffLpq{2,1} := \langle \ADiffLpq{2,1}, \lSigmaMin, \lSigmaPl \rangle.
\end{align*}
Then again $\RotSub$ is the identity path component of each of them.
Moreover,
\begin{align*}
    &\lSigmaPl^2=\lSigmaMin^4=\id,
    &&\lSigmaPl\lSigmaMin\lSigmaPl=\lSigmaMin^{-1},
    &&\lSigmaMin^2 = \lTau,
    &&\lSigmaPl\lSigmaMin = \lTheta,
    &&\lTheta\lTau=\lTau\lTheta,
\end{align*}
which imply that $\AFolDiffLpq{2,1} :=
\langle \RotSub, \lTheta, \lTau, \lSigmaMin, \lSigmaPl \rangle = \langle \RotSub, \lSigmaMin, \lSigmaPl \rangle$, and
\begin{align*}
    &\pi_0\ADiffLpq{2,1} = \langle \lTheta, \lTau \rangle = \bZ_2\times\bZ_2,
    &&\pi_0\AFolDiffLpq{2,1} = \langle \lSigmaMin, \lSigmaPl \rangle \cong \Dih(\bZ_4) = \bD_{4}.
\end{align*}

\item
In all other cases $p>2$ and $\Diff(\Lpq{p}{q})$ still contains the diffeomorphism $\diflpdir{\lTau}{\hTau}{\hTau}$.
Then we put
\begin{align*}
    &\ADiffLpq{p,q} := \langle \RotSub, \lTau\rangle \cong \Dih(\RotSub),
    &&\AFolDiffLpq{p,q} := \langle \ADiffLpq{p,q}, \lSigmaMin, \lSigmaPl\rangle.
\end{align*}
Note that $\RotSub$ is the identity path component of each of these groups.
Let us specify them more precisely.

\begin{itemize}[leftmargin=*]
\item
Suppose $q^2 = 1(\bmod\ p)$, so $\diflprew{\lSigmaPl}{\id}{\id}$ is defined.
Then $\lSigmaPl\circ\lTau=\lTau\circ\lSigmaPl$, whence
\begin{align*}
    &\AFolDiffLpq{p,q} = \langle \RotSub, \lTau, \lSigmaPl \rangle,
    &&\pi_0\AFolDiffLpq{p,q} = \bZ_2 \times \bZ_2.
\end{align*}

\item
Suppose $q^2 = -1 (\bmod\ p)$, so $\diflprew{\lSigmaMin}{\hLambda}{\hMu}$ is defined.
Then $\lSigmaMin^2 = \lTau$, whence
\begin{align*}
    &\AFolDiffLpq{p,q} = \langle \RotSub, \lSigmaMin \rangle,
    &&\pi_0\AFolDiffLpq{p,q} = \bZ_4.
\end{align*}

\item
Otherwise, $q^2\not= \pm 1(\bmod\ p)$, so neither $\lSigmaMin$ nor $\lSigmaPl$ are defined, and
\begin{align*}
    &\AFolDiffLpq{p,q} = \ADiffLpq{p,q} = \langle \RotSub, \lTau\rangle \cong \Dih(\RotSub),
    &&\pi_0\AFolDiffLpq{p,q} = \pi_0\ADiffLpq{p,q} = \bZ_2.
\end{align*}
\end{itemize}
\end{enumerate}
\end{enumerate}

\subsection*{Smale conjecture}
Recall that \term{Smale conjecture} for a manifold $\Mman$ with a ``good'' Riemannian metric is a statement that the inclusion $\Isom(\Mman) \subset \Diff(\Mman)$ of the group of isometries of $\Mman$ into the group of all its diffeomorphisms is a homotopy equivalence.

It is also well known that each lens space $\Lpq{p}{q}$ (except for $\Lpq{0}{1}=S^1\times S^2$) can be described a the quotient of the unit $3$-sphere $S^3$ in $\bC^2$ by a free action of $\bZ_p$ generated by the diffeomorphism
\[
    \delta_{p,q}\colon S^3\to S^3,
    \qquad
    \delta_{p,q}(z_{1},z_{2}) = (e^{2\pi i/p}\cdot z_{1}, e^{2\pi iq/p}\cdot z_{2}).
\]
In particular, $\Lpq{p}{q}$ has a natural metric, called \term{elliptic}, induced from the standard metric on $S^3$ via the corresponding covering map $S^3\to\Lpq{p}{q}$.
The group of isometries $\Isom(\Lpq{p}{q})$ of this metric is described in~\cite[Theorem~2.3]{KalliongisMiller:KMJ:2002}, which, in particular, implies that $\ADiffLpq{p,q} \subset \AFolDiffLpq{p,q} \subset\Isom(\Lpq{p}{q})$, so the above groups consist of isometries, and these three groups coincide exactly when neither $\lSigmaMin$ nor $\lSigmaPl$ is defined.
Thus,
\begin{equation}\label{equ:Isom_Lpq__no_sigma}
    \ADiffLpq{p,q} = \AFolDiffLpq{p,q} = \Isom(\Lpq{p}{q})
    \qquad \Leftrightarrow \qquad
    p>2 \ \text{and} \ q^2\not= \pm 1(\bmod\ p).
\end{equation}

Note further that the Smale conjecture is proved
\begin{enumerate}[leftmargin=*, label={(\alph*)}]
\item for $\Lpq{1}{0}=S^3$, so the inclusion $\Ort(4) \subset \Diff(S^3)$ is a homotopy equivalence, A.~Hatcher~\cite{Hatcher:AnnM:1983};
\item
for all lens spaces $\Lpq{p}{q}$ with $p>2$ in the book~\cite{HongKalliongisMcCulloughRubinstein:LMN:2012};
\item
for all lens spaces except for $S^1\times S^2$ by another methods used Ricci flow in the preprint by R.~Bamler and B.~Kleiner~\cite{BalmerKleiner:RF:2019};
\end{enumerate}
For completeness, let us mention that A.~Hatcher~\cite{Hatcher:ProcAMS:1981} proved that $\Diff(S^1\times S^2)=\Diff(\Lpq{0}{1})$ has the homotopy type of $\Omega(\Ort(3))\times\Ort(3) \times\Ort(2)$.

\subsection*{Polar Morse-Bott foliation on $\Lpq{p}{q}$}
Recall that in Section~\ref{sect:main_result:all} we defined a Morse-Bott foliation $\AFoliation$ on the solid torus $\ATor$ by the central circle and $2$-tori ``parallel'' to the boundary.
In particular, $\partial\ATor$ is a leaf of $\AFoliation$.
Since $\gxi$ identifies $\partial\ATor\times\{0\}$ with $\partial\ATor\times\{1\}$, it induces a certain foliation $\AFoliation_{p,q}$ on $\Lpq{p}{q}$ whose leaves are the images of the corresponding leaves of foliations on those tori.
This foliation will be called \term{polar}.
In particular, $\AFoliation_{p,q}$ has two singular leaves (the central circles) $\Cman_0$ and $\Cman_1$ and all other leaves are $2$-tori parallel each other.

Since the above diffeomorphisms $\hRho{\alpha,\beta}$, $\hDtwist$, $\hLambda$, $\hMu$, $\hTau$ of $\ATor$ belong to $\DiffAFolLp$, it follows that $\ADiffLpq{p,q} \subset \FolLpDiffLpq{p}{q}$, while $\lSigmaPl,\lSigmaMin\in\FolDiffLpq{p}{q}\setminus\FolDiffPlLpq{p}{q}$, and thus we get the following inclusions of subgroups:
\begin{equation}
\begin{gathered}\label{equ:diag:Lpq_all_groups}
\xymatrix@R=1.2em{
    \FolLpDiffLpq{p}{q}   \ar@{^(->}[r] &
    \FolDiffPlLpq{p}{q}   \ar@{^(->}[r] &
    \FolDiffLpq{p}{q}\\
    \ADiffLpq{p,q} \ar@{^(->}[rr] \ar@{^(->}[u] && \AFolDiffLpq{p,q}  \ar@{^(->}[u]
}
\end{gathered}
\end{equation}
By Theorem~\ref{th:DiffLpq_fol_homtype_a}, $\FolLpDiffLpq{p}{q}$ is a strong deformation retract of $\FolDiffPlLpq{p}{q}$.
Moreover, we also have that:
\begin{theorem}[{\rm\cite{KhokhliukMaksymenko:JHRS:2023}}]
\label{th:hom_type__FolDiffLpq_lp}
The inclusion $\ADiffLpq{p,q} \subset \FolLpDiffLpq{p}{q}$ is a weak homotopy equivalence.
\end{theorem}
\begin{lemma}\label{lm:no_sigma_pm}
The following conditions are equivalent:
\begin{enumerate}[label={\rm(\alph*)}]
\item\label{enum:Lpq_exch_tori:no_diff}
there are no diffeomorphisms $\dif\in\Diff(\Lpq{p}{q})$ exchanging $\ATor_0$ and $\ATor_1$;
\item\label{enum:Lpq_exch_tori:p}
$p>2$ and $q^2\not=\pm1(\bmod\ p)$, i.e.\ exactly when neither $\lSigmaMin$ nor $\lSigmaPl$ are defined;
\item\label{enum:Lpq_exch_tori:Lgroups}
$\AFolDiffLpq{p,q}=\ADiffLpq{p,q}$;
\item\label{enum:Lpq_exch_tori:fol}
$\FolDiffPlLpq{p}{q}=\FolDiffLpq{p}{q}$, i.e.\ every $\dif\in\FolDiffLpq{p}{q}$ leaves invariant each $\Cman_0$ and $\Cman_1$.
\end{enumerate}
\end{lemma}
\begin{proof}
In fact, the equivalence~\ref{enum:Lpq_exch_tori:no_diff}$\Leftrightarrow$\ref{enum:Lpq_exch_tori:p} is well known,
\ref{enum:Lpq_exch_tori:p}$\Leftrightarrow$\ref{enum:Lpq_exch_tori:Lgroups} follows from the definition of $\AFolDiffLpq{p,q}$, and the implication~\ref{enum:Lpq_exch_tori:fol}$\Rightarrow$\ref{enum:Lpq_exch_tori:p} is evident.

\ref{enum:Lpq_exch_tori:no_diff}$\Rightarrow$\ref{enum:Lpq_exch_tori:fol}.
Let $\Mman = \Lpq{p}{q}\setminus(\Cman_0\cup\Cman_1) = \gfunc^{-1}\bigl((0;2)\bigr)$.
Then $\gfunc$ has no critical points in $\Mman$ and each level set $\gfunc^{-1}(t)$, $t\in(0;2)$, is diffeomorphic to a $2$-torus.
Moreover, using a standard technique with some gradient flow of $\gfunc$, one can construct a diffeomorphism  $\psi\colon\Mman\to T^2\times(0;2)$ such that $\func\circ\psi^{-1}(\py,t) = t$, $t\in(0;2)$.

Now suppose that~\ref{enum:Lpq_exch_tori:fol} fails, so there exists $\dif\in\FolDiffLpq{p}{q}\setminus\FolDiffPlLpq{p}{q}$.
If $\dif(\gfunc^{-1}(1))=1$, then $\dif$ exchanges $\ATor_0$ and $\ATor_1$, so~\ref{enum:Lpq_exch_tori:no_diff} also fails.
Suppose $\dif(\gfunc^{-1}(1)) = \gfunc^{-1}(\tau)$ for some $\tau\in(0;2)$ and let $\mu\colon (0;2)\to(0;2)$ be a diffeomorphism fixed near $0$ and $2$ and such that $\mu(\tau)=1$, so it gives the diffeomorphism $k'\colon T^2\times(0;2)\to T^2\times(0;2)$, $k'(\py,t)=k(\py,\mu(t))$.
Notice also that $\dif(\Mman)=\Mman$, whence we have a diffeomorphism $k = \psi\circ\dif\circ\psi^{-1}\colon T^2\times(0;2) \to T^2\times(0;2)$ such that $k(T^2\times 1) = T^2\times \tau$.
Define the following diffeomorphism $\dif'\in\Diff(\Lpq{p}{q})$ by
\[
    \dif'(\px)=
    \begin{cases}
        \psi^{-1}\circ k'\circ k\circ\psi(\px), & \px\in\Mman,\\
        \dif(\px),                              & \px\in\Cman_0\cup\Cman_1.
    \end{cases}
\]
Then $\dif'\in\FolDiffLpq{p}{q}\setminus\FolDiffPlLpq{p}{q}$ coincides with $\dif$ near $\Cman_0\cup\Cman_1$ and leaves invariant $\gfunc^{-1}(1)$.
\end{proof}

\begin{theorem}\label{th:DiffLpq_fol_homtype}
{\rm 1)}~Suppose $p>2$ and $q^2\not=\pm1(\bmod\ p)$.
Then we have the following commutative diagram of inclusions in which the arrows denoted by (w.)h.e. are (weak) homotopy equivalences (and new statements of this paper are written in bold):
\[
\xymatrix@C=6em{
   \FolLpDiffLpq{p}{q}
   \ar@{^(->}[r]^-{\textbf{h.e.}}_-{\text{Theorem~\ref{th:DiffLpq_fol_homtype_a}}} &
   \FolDiffPlLpq{p}{q} \ar@{=}[r]_-{\text{Lemma~\ref{lm:no_sigma_pm}}} &
   \FolDiffLpq{p}{q}
   \ar@{^(->}[d]^-{\textbf{w.h.e.}}  \\
   \ADiffLpq{p,q} = \AFolDiffLpq{p,q} = \langle\RotSub,\lTau\rangle
   \ar@{^(->}[u]^-{\text{w.h.e.}}_-{\text{Theorem~\ref{th:hom_type__FolDiffLpq_lp}}}
   \ar@{=}[r]_-{\text{\cite[Theorem~2.3]{KalliongisMiller:KMJ:2002}}} &
   \Isom(\Lpq{p}{q})
   \ar@{^(->}[r]^-{\text{h.e.}}_-{\text{Smale conjecture}} &
   \Diff(\Lpq{p}{q})
}
\]
In particular, each of those groups has the homotopy type of the disjoint union of two $2$-tori $\RotSub\sqcup\RotSub$.

{\rm 2)}~For all other values of $(p,q)$, i.e.\ when either $\lSigmaMin$ or $\lSigmaPl$ is defined, the inclusion of pairs
\begin{equation}\label{equ:incl_Lgroups_Folgroups}
\bigl( \AFolDiffLpq{p,q}, \ADiffLpq{p,q} \bigr)
\ \subset \
\bigl( \FolDiffLpq{p}{q}, \ \FolDiffPlLpq{p}{q} \bigr)
\end{equation}
is a weak homotopy equivalence.
In particular, each path component of each of those groups is also homotopy equivalent to a $2$-torus $\RotSub$, and we also have the following diagram:
\begin{equation}
    \begin{gathered}\label{equ:diag:Lpq_all_groups1}
    \xymatrix@R=1.2em{
        \FolLpDiffLpq{p}{q}   \ar@{^(->}[r]^-{h.e.} &
        \FolDiffPlLpq{p}{q}   \ar@{^(->}[r] &
        \FolDiffLpq{p}{q}\\
        \ADiffLpq{p,q} \ar@{^(->}[rr] \ar@{^(->}[u]^-{w.h.e.} && \AFolDiffLpq{p,q}  \ar@{^(->}[u]_-{w.h.e.}
    }
    \end{gathered}
\end{equation}
in which non-marked arrows are inclusions of index $2$ subgroups.
\end{theorem}
\begin{proof}
1) The left vertical and all horizontal arrows are explained at the diagram.
They imply that the right vertical arrow is also a weak homotopy equivalence.

2) By Lemma~\ref{lm:no_sigma_pm}, $\AFolDiffLpq{p,q} \not= \ADiffLpq{p,q}$ and $\FolDiffLpq{p}{q} \not= \FolDiffPlLpq{p}{q}$.
Hence, $\AFolDiffLpq{p,q}$, resp.\ $\FolDiffLpq{p}{q}$, is a union of two adjacent classes by the group $\ADiffLpq{p,q}$, resp.\ $\FolDiffPlLpq{p}{q}$.
By 1) the inclusion $\sincl\colon\ADiffLpq{p,q}\subset\FolDiffPlLpq{p}{q}$ is a weak homotopy equivalence, i.e.\ $\sincl\colon\pi_0\ADiffLpq{p,q}\to\pi_0\FolDiffPlLpq{p}{q}$ is a bijection, and $j$ yields a weak homotopy equivalence between the corresponding identity path components.
Hence, so must be the inclusion of pairs~\eqref{equ:incl_Lgroups_Folgroups}.
\end{proof}

\subsection*{Awknowledgement}
The author is sincerely grateful to anonymous Referee for very useful remarks and suggestions which allowed to improve the contents of the paper and exposition.

% \bibliographystyle{plainurl}
% \bibliography{biblio}

\end{document}